\documentclass[12pt,english]{article}

\usepackage[T1]{fontenc}
\usepackage[latin9]{inputenc}
\usepackage{babel}
\usepackage{cancel}

\usepackage{xcolor}

\usepackage{microtype}

\usepackage[
tmargin=3cm,
bmargin=3cm,
lmargin=3.5cm,
rmargin=3.5cm]{geometry}

\usepackage[
style=authoryear,
firstinits=true,
style=numeric-comp,
sorting=none,
url=false,
doi=false,
isbn=false,
backend=biber]{biblatex}
\usepackage{verbatim}
\usepackage{mathrsfs}
\usepackage{amsmath}
\usepackage{dsfont}

\usepackage{enumerate}
\makeatletter
\newenvironment{labellist}[1]
	{\begin{list}{}
		{\settowidth{\labelwidth}{#1}
		 \setlength{\leftmargin}{\labelwidth}
		 \addtolength{\leftmargin}{\labelsep}
		 }}
	{\end{list}}
\makeatother

\usepackage{amsthm}
\usepackage{amssymb}
\usepackage{iftex}

\providecommand{\corollaryname}{Corollary}
\providecommand{\definitionname}{Definition}
\providecommand{\lemmaname}{Lemma}
\providecommand{\remarkname}{Remark}
\providecommand{\theoremname}{Theorem}

\theoremstyle{plain}
\newtheorem*{statement}{Main result}
\newtheorem{thm}{\protect\theoremname}[section]
\newtheorem{cor}[thm]{\protect\corollaryname}
\newtheorem{lem}[thm]{\protect\lemmaname}

\theoremstyle{definition}
\newtheorem{defn}[thm]{\protect\definitionname}

\theoremstyle{remark}
\newtheorem{rem}[thm]{\protect\remarkname}

\usepackage{graphicx}
\usepackage{scalerel}

\newcommand{\inter}{i}%
\newcommand{\exter}{o}%
\newcommand{\tcomp}{\mathrm{{\scriptscriptstyle T}}}%
\newcommand{\zmindex}{\mathrm{{\scriptscriptstyle ZM}}}%
\newcommand{\dindex}{\mathrm{{\scriptscriptstyle D}}}%
\newcommand{\bindex}{\mathrm{{\scriptscriptstyle B}}}%
\newcommand{\sindex}{\mathrm{{\scriptscriptstyle S}}}%
\newcommand{\gindex}{\mathrm{{\scriptscriptstyle G}}}%

\newcommand{\conlim}{\substack{\X\to\P \\ \X\in\Gamma(\P)}}
\newcommand{\conlimi}{\substack{\X\to\P \\ \X\in\Gamma_{\inter}(\P)}}
\newcommand{\conlimo}{\substack{\X\to\P \\ \X\in\Gamma_{\exter}(\P)}}

\newcommand{\plus}{+}
\newcommand{\minus}{-}

\renewcommand{\o}{\Omega}%
\newcommand{\bo}{\partial\Omega}%
\newcommand{\bbo}{(\bo)}%

\newcommand{\measure}[1]{\;d\sigma\negmedspace\left(#1\right)}%
\newcommand{\pv}{\text{\ensuremath{\mathrm{p.v.}} }}%

\newcommand{\Q}{q}%
\renewcommand{\P}{p}%
\newcommand{\X}{x}%

\renewcommand{\tan}[1]{\text{T}_{#1}(\bo)}
\newcommand{\reg}{\mathrm{\mathop{Reg}}\bbo}%
\newcommand{\regb}{\mathrm{\mathop{Reg}}(\mathbb{B}_{j})}%

\newcommand{\ltd}[1]{L^{2}(\bo ; \mathbb{R}^{#1})}%
\newcommand{\lt}{L^{2}\bbo}%
\newcommand{\ltz}{L_{\zmindex}^{2}\bbo}%
\newcommand{\sob}{W^{1,2}\bbo}%
\newcommand{\divf}{D\bbo}%
\newcommand{\grad}{G\bbo}%
\newcommand{\tang}{T\bbo}%
\newcommand{\tangz}{T_{\negthinspace\zmindex}\bbo}%
\newcommand{\hardyp}{H^{2}_{\hspace{-0.5pt}\plus}\bbo}%
\newcommand{\hardym}{H^{2}_{\!\minus}(\bo)}%
\newcommand{\Iod}{I\bbo}%
\newcommand{\iOd}{O\bbo}%

\newcommand{\pot}{\mathcal{P}}%
\newcommand{\poti}{\mathcal{P}_{\inter}}%
\newcommand{\poto}{\mathcal{P}_{\!\exter}}%

\newcommand{\tri}{B_{\hspace{0.5pt}\inter}^{\phantom{\star}}}%
\newcommand{\tric}{B_{\hspace{0.5pt}\inter}^{\star}}%
\newcommand{\tro}{B_{\exter}^{\vphantom{\star}}}%
\newcommand{\troc}{B_{\exter}^{\star}}%

\newcommand{\projs}{P_{\negthickspace\sindex}}%
\newcommand{\projt}{P_{\negthickspace\bindex}^{\vphantom{\perp}}}%
\newcommand{\projtp}{P_{\negthickspace\bindex}^{\perp}}%
\newcommand{\restri}{R_{\zmindex}}%
\newcommand{\projp}{P_{\negthickspace\plus}}%
\newcommand{\projm}{P_{\negthickspace\minus}}%
\newcommand{\projd}{P_{\negthickspace\dindex}}%
\newcommand{\projD}{\mathbb{P}_{\negthickspace\dindex}}%
\newcommand{\projDp}{\projD^{\perp}}%
\newcommand{\projv}{P_{\inter}}%
\newcommand{\projc}{P_{\exter}}%
\newcommand{\projz}{P_{\negthickspace\zmindex}}%
\newcommand{\projg}{P_{\negthickspace\gindex}}%

\newcommand{\nablat}{\nabla_{\negthickspace\tcomp}}%
\newcommand{\DLay}{\mathscr{K\!}}%
\newcommand{\SLay}{\mathscr{S\negmedspace}}%
\newcommand{\dlay}{K}%
\newcommand{\slay}{S}%

\newcommand{\divt}{\mathrm{\mathop{div_{\tcomp}}}}%

\newcommand{\consti}{\nu_{\mathrm{o}}}%

\usepackage[colorinlistoftodos,prependcaption,textsize=tiny,textwidth=2.5cm]{todonotes}

 \usepackage{authblk}

\title{\textbf{\large Decomposition of $L^{2}$-vector fields on Lipschitz surfaces: characterization via null-spaces of the scalar  potential}}

\author{L. Baratchart$^{1}$, C. Gerhards$^{2}$, and A. Kegeles$^{2}$\footnote{alexander.kegeles@geophysik.tu-freiberg.de}}
\date{\textit{\footnotesize $^{1}$INRIA, Project APICS, 2004 route de Lucioles, BP 93, }\\
\textit{\footnotesize Sophia-Antipolis F-06902 Cedex, France}\bigskip{}
\\
\textit{\footnotesize $^{2}$TU Bergakademie Freiberg, Institute of Geophysics and Geoinformatics,}\\ 
\textit{\footnotesize Gustav-Zeuner-Str. 12, 09599 Freiberg, Germany}}

\begin{document}
\maketitle

\begin{abstract}
For $\bo$ the boundary of a bounded and connected strongly Lipschitz domain in $\mathbb{R}^{d}$ with  $d\geq3$, we prove that any field $f\in\ltd d$ decomposes, in an  unique way, as the sum of three silent vector fields---fields whose magnetic potential vanishes in one or both components of $\mathbb{R}^d\setminus\bo$. Moreover, this decomposition is orthogonal if and only if $\bo$ is a sphere.
We also show that any $f$ in $\ltd d$ is uniquely the sum of two silent fields and a Hardy function, in which case the sum is orthogonal regardless of $\bo$; we express the corresponding orthogonal projections in terms of layer potentials. 
When $\bo$ is a sphere, both decompositions coincide and match
what has been called the Hardy-Hodge decomposition in the literature.
\end{abstract}

\section{Introduction}
Orthogonal direct sum decompositions provide an important tool in harmonic analysis and applied sciences. Orthogonal direct sums are closely related to the method of orthogonal projections, which was introduces into potential theory by Weyl \cite{weyl1940}. Further developed by Vishik \cite{vishik1949method, vishik1951strongly} and G\aa rding \cite{gaarding1953dirichlet}, and combined with the results by Lax and Milgram \cite{lax2005parabolic} this method quickly developed and serves now as a powerful approach to elliptic and parabolic boundary value problems (see for example \cite{friedman2008partial} or \cite{ hormander2015analysis}). On a sphere, there are two sums that are particularly interesting for magnetic inverse problems: one is the decomposition of fields into contributions that do not create a magnetic field inside or outside of the sphere; the other is the decomposition of fields into contributions that can be harmonically continued inside and outside the sphere. In this paper, we show how these sums generalize to hold on Lipschitz surfaces. 

The first decomposition appears in the following setting: Consider a ball whose very thin boundary layer is magnetized. The magnetization of that layer defines a vector field supported near the boundary of the ball. Idealizing, we can say that this vector field is defined on a sphere. Such a magnetization will generate a magnetic field inside and outside the sphere. By measuring the surrounding magnetic field we thus can \lq\lq see'' this magnetization. However, it is possible that a magnetization will be invisible---meaning that its magnetic field will vanish at least inside or outside of the sphere. Inverse problems in magneto-statics are non-unique precisely because invisible magnetizations exist, and for a generic shape it is important to understand when a magnetization can be invisible. In this paper we answer this question for a fairly general class of Lipschitz surfaces.

For a sphere, it is known that every sufficiently regular vector field  (say, of $L^2$-class) is a sum of three invisible magnetizations: loosely said, one that is invisible inside; one that is invisible outside; and one that is invisible everywhere. To make this statement more precise let $\o$ be an open region with the boundary $\bo$ and consider the following spaces:
\begin{labellist}{0.0.000}
\setlength\itemsep{-3pt}
\item[$\divf$] the space of square integrable fields that are invisible everywhere;
\item[$\Iod$] the space, orthogonal to $\divf$, of fields invisible only inside the region $\o$, and the zero field;
\item[$\iOd$] the space, orthogonal to $\divf$, of fields invisible only outside the region $\o$, and the zero field.
\end{labellist}
If $\bo$ is a sphere, $\Iod$ and $\iOd$ are mutually orthogonal and the space of $\mathbb{R}^{3}$-valued square integrable fields splits into an orthogonal direct sum
\begin{align}
	\ltd 3 = \Iod + \iOd + \divf. \label{eq:IOZ decomposition}
\end{align}
Additionally, the space $\Iod$ defines precisely those magnetizations that one can \lq\lq see'' from the outside. Because of this, decomposition \eqref{eq:IOZ decomposition} becomes an important tool to solve magnetic inverse problems on a sphere. For example,  in medical imaging it is used to process EEG/MEG measurements of the human head models \cite{lewmicfok20}; in scanning magnetic microscopy it appears in the study of planetary rock samples \cite{lima13}; it has been used to separate magnetic fields measured on satellite orbits with respect to their sources \cite{backus96,mayermaier06,olsen10b,baratchartgerhards16};  and to invert magnetic fields for spherical source currents \cite{mayer04} as well as for the lithospheric magnetization \cite{gerhards16a,gubbins11,verles19,vervelidou16}.

The second important decomposition on the sphere is the so called Hardy-Hodge decomposition 
\begin{align}
	\ltd d = \hardyp + \hardym + D_{f}(\bo); \label{eq:hhd}
\end{align}
see \cite{baratchart13,baratchart17a,gerhards16a} for the Hardy-Hodge decomposition on a plane or a sphere in various smoothness classes. The space $D_{f}(\bo)$ is the space of divergence-free fields. The spaces $\hardyp$ and $\hardym$ in \eqref{eq:hhd} are the Hardy-spaces, initially introduced on half spaces in \cite{stein1960},
studied over $C^1$-hypersurfaces in \cite{fabes1981,coifman1977}, and on Lipschitz surfaces in \cite{ver84,dahl87}. In a nutshell, these spaces define vector fields that have a harmonic extension inside or outside of the sphere.
Because harmonic functions have much stronger properties than merely square integrable functions, the Hardy-Hodge decomposition is a powerful tool in mathematical analysis. In fact it features the Calderon-Zygmund operator, which gave rise to the elliptic regularity theory \cite{stein70, steinweiss71}. Eventually, on the sphere it holds that $\Iod = \hardyp$ and $\iOd = \hardym$ which again makes \eqref{eq:hhd} pertinent to various inverse problems.

If $\bo$ is not a sphere, the decomposition \eqref{eq:IOZ decomposition} seems new. Also the Hardy-Hodge decomposition seems unknown if $\bo$ is Lipschitz (see \cite{BPT} for recent results). Consequently, also the relation between the spaces $\Iod$, $\iOd$ and $\hardyp$, $\hardym$ is missing. In this paper, we address these questions for Lipschitz surfaces and vector fields of $L^{2}$-class. We believe that the answer will allow one
to extend existing analytical techniques known on the sphere to Lipschitz surfaces. Also it could improve the accuracy of computations in cases where the sphere would only be an approximation of the actual underlying geometry. 

The main result of this paper is to prove the following statement that we have separated into several theorems for clarity:
\begin{statement}
Let $\o$ be a bounded and connected Lipschitz domain in $\mathbb{R}^{d} \ \left(d\geq 3\right)$ with a connected boundary $\bo$. The space of square integrable $\mathbb{R}^{d}$-valued vector fields on $\bo$ decomposes into the following orthogonal direct sums
\begin{align*}
	\ltd d &= \Iod \oplus \hardym \oplus \divf  &&(\text{Theorem \ref{thm:L2 decomposition in}})\\
	\ltd d &= \hardyp \oplus \iOd \oplus \divf  &&(\text{Theorem \ref{thm:L2 decomposition out}});
\end{align*}
and into the following topological direct sums
\begin{align*}
	\ltd d &= \hardyp + \hardym + D_{f}(\bo) &&(\text{Theorem \ref{thm:hardy-hodge decomposition}}) \\
	\ltd d &= \Iod + \iOd + \divf &&(\text{Theorem \ref{thm:ahh}}).
\end{align*}
These topological direct sums are orthogonal if and only if $\bo$ is a sphere (Theorem \ref{thm:HardyOrto} and Corollary \ref{cor:IODortho}); and in this case $\hardyp = \Iod$ and $\hardym = \iOd$ (Corollary \ref{cor:ieh}). Moreover, regardless of $\bo$ it always holds that $D_{f}(\bo) = \divf$.
\end{statement}

In Section \ref{sec:prelim} we set up notational conventions and discuss layer potentials. The latter are well-studied in the literature (for example, \cite{dahl87,gilmur91,ver84}) and lie at the core of all arguments in this paper. Although the main physical interest attaches to surfaces in $\mathbb{R}^3$, it would be artificial to restrict to dimension 3 and for that reason we present the material in $\mathbb{R}^d$ for $d\geq3$.

\section{Preliminaries and Notation}\label{sec:prelim}

\subsection{Conventions}

In this paper we reserve the symbol $\o$ to denote a bounded and connected strongly Lipschitz domain in the $d$-dimensional Euclidean space $\mathbb{R}^{d}$  $(d\geq3)$. Throughout, we assign the symbol $\bo$ to denote the boundary of $\o$, which is a connected and closed hypersurface, locally given as the graph of a Lipschitz function (see appendix \ref{appendix} for more details). We call $\o$ the inside of $\bo$ and $\o^{o}=\mathbb{R}^{d}\backslash\overline{\o}$ the outside. The surface measure on $\bo$ will be denoted by $\sigma$; it is the restriction to $\bo$ of the $(d-1)$-Hausdorff measure; since $\partial\Omega$ is compact, $\sigma$ is finite. Statements made almost everywhere (a.e.) on $\bo$ are always understood with respect to $\sigma$.

On $\mathbb{R}^{d}$, we denote the Euclidean scalar product by $\left\langle x, y\right\rangle_{\mathbb{R}^{d}}$ and the Euclidean norm by $|x|=\sqrt{\langle x,x\rangle_{\mathbb{R}^{d}}}$. If $f$ is a differentiable function defined on an open region in $\mathbb{R}^{d}$, then $\nabla f$ denotes the Euclidean gradient of $f$. We denote the Euclidean gradient by $\Delta$ that reads in coordinates $\Delta f = \sum_{i=1}^{d} \partial^{2}f/\partial x_{j}^{2}$.

The space $\lt\left(=L^{2}\left(\bo,\sigma\right)\right)$ comprises scalar-valued functions on $\bo$ that are square integrable with respect to $\sigma$. As usual, we identify functions that coincide a.e. on $\bo$. The subspace $\ltz\subset \lt$ consists of functions with zero mean; note that $\lt$-functions  indeed have a well-defined mean on $\bo$, since $\sigma$ is finite. We call an $\mathbb{R}^{d}$-valued function on $\bo$ a field, and we denote the space of square integrable fields by $\ltd d\left(=L^{2}\left(\bo;\mathbb{R}^{d},\sigma\right)\right)$. If $f$ and $g$ are two fields in $\ltd d$, their scalar product is $\left\langle f,g\right\rangle =\int_{\bo}\left\langle f(\Q),g(\Q)\right\rangle_{\mathbb{R}^{d}}\measure{\Q}$ and the norm of $f$ in $\ltd d$ is $\|f\|=\sqrt{\left\langle f,f\right\rangle}$. Clearly, $\ltd d$ is a Hilbert space.

The tangent space to $\bo$ at $\X$ is denoted by $\tan{\X}$; it is well-defined a.e. and so is the outer unit normal field $\eta$ (see appendix \ref{appendix} for more details).  We identify $\tan{\X}$ with a $(d-1)$-dimensional hyperplane in $\mathbb{R}^{d}$. If $f$ is a field in $\ltd d$ and for a.e. $\X$ the vector $f(\X)$ lies in $\tan{\X}$, we call $f$ a tangent field. We denote the space of all tangent fields by $\tang$; it is a closed subspace of $\ltd d$. We will often split a field $f\in \ltd d$ into a normal part $f_{\eta}\in\lt$ and the tangent part $f_{\tcomp}\in\tang$, and write $f=\eta f_{\eta} + f_{\tcomp}$.

A Lipschitz function $f:\bo\to\mathbb{R}$ is differentiable a.e. on $\bo$ and therefore it has a well-defined tangential gradient $\nablat f(\X)\in \tan{\X}$ at almost every $\X$. We define  the Sobolev space $\sob$ on $\bo$ to be the completion of Lipschitz functions for the norm $(\|f\|^2+\|\nablat f\|^2)^{1/2}$. If $f$ is in $\sob$,  {\it a fortiori} $f$ is in $\lt$ and $\nablat f$ is a tangent field in $\tang$ (see appendix \ref{appendix} for details).

If $A$ is a bounded operator then $\mathcal{N}\left[A\right]$ denotes its null space and $\mathcal{R}\left[A\right]$ its range. For  $N$ a subspace of a Hilbert space, we let $N^{\perp}$ designate its orthogonal complement. When saying that $A$ is invertible, we always mean that $A$ has a bounded inverse.

\paragraph{Regular family of cones.}
We will use a regular family of cones. Even though this family does not explicitly appear in what follows, it is fundamental to the very definition of boundary values on $\bo$ and the limiting behavior of potentials that we invoke repeatedly. The existence of a regular family of cones is folklore, but it is hard to locate a proof, see \cite{BPT}.

More precisely, for $\theta\in(0,\pi/2)$ and $y,z\in\mathbb{R}^{d}$ with $|z|=1$, we put $C_{\theta,z}(y)$ for the open, right circular, positive cone with vertex at $y$, axis directed by $z$, and  aperture angle $2\theta$; the cone being truncated to some fixed suitable length. We do not make the length explicit in the notation, for it  plays no role provided that it is small enough (how small depends on $\Omega$). To each point of $\bo$ we can attach two  ``natural'' cones with fixed aperture such that: (i) their direction is that of the graph or opposite to it; (ii) their length is small enough that one of them lies in $\Omega$ and the other in $\o^{o}$.

Let $\mathbb{S}^{d-1}$ denote the $(d-1)$-dimensional unit sphere and  $Z\colon \bo\to\mathbb{S}^{d-1}$ be a continuous function with the following property: for some $\theta_{1}<\theta<\theta_{2}$ independent of $\Q$, we require that the cone $C_{\theta,\pm Z(\Q)}(\Q)$, truncated to suitable length independent of $\Q$, contains a natural cone of aperture $2\theta_{1}$ and is contained in another natural cone of aperture $2\theta_{2}$. We may assume, replacing $Z$ by $-Z$ if necessary, that $C_{\theta,Z(\Q)}(\Q)\subset\Omega$ and $C_{\theta,-Z(\Q)}(\Q)\subset\o^{o}$.

A regular family of cones for $\Omega$ (resp. $\o^o$) is a map that associates to every $\Q\in\bo$ a cone $C_{\theta, Z(\Q)}(\Q)\subset\Omega$ (resp. $C_{\theta,- Z(\Q)}(\Q)\subset\o^o$) with $Z$ as above, compare the definition in \cite{ver84}. Hereafter, we fix such a family once and for all, and we write $\Gamma_{\inter}(\Q)$ for the inner cone at $\Q$ in this family, $\Gamma_{\exter}(\Q)$ for the outer cone at $\Q$, and $\Gamma(\Q)=\Gamma_{\inter}(\Q)\cup \Gamma_{\exter}(\Q)$ for the double cone.

Associated to the regular family of cones is a nontangential maximal function defined as follows. If $g_\inter\colon\o\to\mathbb{R}^d$ (resp. $g_\exter\colon\o^{\exter}\to\mathbb{R}^d$ ) is a function defined on $\o$ (resp. $\o^{\exter}$), we denote the nontangential maximal functions of $g$ at $\P\in\bo$ as,
\begin{align}
	g_{i}^{M}(\P) & =\sup\left\{ \left|g(\X)\right|\,:\,\X\in\Gamma_{\inter}(\P)\right\},\\
	\mathrm{(resp.\ } 
	g_{\exter}^{M}(\P) & =\sup\left\{ \left|g(\X)\right|\,:\,\X\in\Gamma_{\exter}(\P)\right\}\ 	\mathrm{)}.
\end{align}
When $g$ is defined on $\o\cup\o^{\exter}$, we set
\begin{align}
	g^{M}(\P) & =\sup\left\{ \left|g(\X)\right|\,:\,\X\in\Gamma(\P)\right\}.
\end{align}

A function $h$ on an open set $O\subset\mathbb{R}^d$ 
is harmonic if it satisfies $\Delta h=0$.  As soon as $\Delta h$ exists in
the distributional sense, $h$ is infinitely differentiable.
It follows from \cite[Sec. 5, thm.]{HuWh68} and \cite[thm. 1]{DahlbergHM1977}
that every harmonic function on $\o$ (resp. $\o^{\exter}$) whose maximal function is in $\lt$ has a nontangential limit a.e. on $\bo$, and that limit function is in $\lt$.

\subsection{Potentials}

In this section, we summarize known results about the single and the
double layer potentials. Most of the statements and 
references  or proofs for them can be found in \cite{ver84}. In the following we denote the surface area of a $d$-dimensional sphere by $\omega_{d}$. 

\paragraph*{Single layer potential.}

The single layer potential of $f\in\lt$ is
\begin{equation}
	\SLay f \left(\X\right) \doteq
	\frac{-1}{\omega_{d} 
	\left(d-2\right)} \int_{\bo} \frac{1} {\left|\X-\Q\right|^{d-2}}f\left(\Q\right)\measure{\Q}\qquad\left(\X\in\mathbb{R}^{d}\backslash\bo\right).\label{eq:single layer potential}
\end{equation}
It is harmonic on $\mathbb{R}^{d}\backslash\bo$. The result below follows for instance from
the combination of \cite[thm. 1]{DahlbergWNE80} and \cite[ch. 15, thm. 1]{CoifmanMeyer97}, see also \cite[thm. 08.D and lem. 1.3]{ver84}. 
\begin{lem}
For a.e. $\P\in\bo$, the limit,
\begin{align}
\lim_{\conlim}
\SLay f\left(\X\right) & =\pv\frac{-1}{\omega_{d}\left(d-2\right)}\int_{\bo}\frac{1}{\left|\X-\Q\right|^{d-2}}f(\Q)\measure{\Q}\doteq\slay f\left(\P\right),\label{eq:boundary single layer potential}
\end{align}
exists,  and the limit function is in $\sob$. Moreover, 
$\|(\mathcal{S}f)^M\|\leq C\|f\|$ for some constant $C=C(\bo)$.
\end{lem}
The operator $\slay \colon \lt \to\sob$, defined by the (weakly) singular integral in (\ref{eq:boundary single layer potential}), is a bounded linear operator \cite[lem. 1.8]{ver84}; when no confusion is possible, we also call it the single layer potential.

The following property of
$\slay$ will be important in what follows. 
\begin{thm}[{\cite[thm 3.3]{ver84}}]
\label{thm:Inverstable S}The operator $S\colon\lt \to\sob $
is invertible. 
\end{thm}
\begin{rem}
The previous results hold for a more general range of exponents, but for the 
purpose of this paper statements in $L^2$ will suffice.
\end{rem}

\paragraph{The gradient of $\protect\SLay f$.}

The Euclidean gradient of the single layer potential is,
\begin{equation}
\nabla\SLay f\left(\X\right)\doteq\frac{1}{\omega_{d}}\int_{\bo}\frac{\X-\Q}{\left|\X-\Q\right|^{d}}f(\Q)\measure{\Q}\qquad\left(\X\in\mathbb{R}^{d}\backslash\bo\right).\label{eq:gradient of SLP for scalars}
\end{equation}
Each of the vector components of $\nabla\SLay f $ is a harmonic function on $\mathbb{R}^{d}\backslash\bo$ and it holds that $\|\left(\nabla\SLay f\right)^{M}\|<C\|f\|$ for some constant $C=C(\Omega)$. This last fact follows from \cite[ch. 15, thm. 1]{CoifmanMeyer97} (see also \cite[lem. 1.3]{ver84}). Thus,
$\nabla\mathcal{\SLay}f$ has nontangential limits a.e. on $\bo$ from each side, that define two fields in $T(\bo)$. We shall describe their tangential and normal components separately. We begin with the tangential component, which is the same from either side
(compare \cite[thm 1.6.]{ver84}):

\begin{lem}
For every tangent field $\tau\in\tang$ 
 the limit
\begin{align}
  \lim_{\conlim}\left\langle \tau(\P),\nabla\SLay f\left(\X\right)\right\rangle _{\mathbb{R}^{d}} & =\pv\frac{1}{\omega_{d}}\int_{\bo}\frac{\left\langle \tau(p),\P-\Q\right\rangle _{\mathbb{R}^{d}}}{\left|\P-\Q\right|^{d}}f(\Q)\measure{\Q}
                                                                                                    \nonumber\\
 & =
\left\langle 
\tau(\P), \nablat\slay f \left(\P\right)
\right\rangle _{\mathbb{R}^{d}}
\label{eq:limit of SLP}
\end{align}
exists at a.e. $\P\in\bo$, and defines the tangent field $\nablat\slay f\in \tang$. Moreover, $\|\nablat\slay f\|\leq C \|f\|$.
\end{lem}
\begin{proof}
See Lemma \ref{convgrad} in Appendix \ref{appendix}.
\end{proof}

We can define the operator $\nablat\slay\colon\lt \to\tang$ such that $\left(\nablat\slay\right)f=\nablat \left(\slay f\right)$. By the above lemma, $\nablat\slay$ is bounded and linear.

\paragraph{The scalar potential.}

The scalar potential of a field $f\in\ltd d$ is 
\begin{equation}
\mathcal{P}f\left(\X\right)\doteq\frac{1}{\omega_{d}}\int_{\bo}\frac{\left\langle \Q-\X,f(\Q)\right\rangle _{\mathbb{R}^{d}}}{\left|\Q-\X\right|^{d}}\measure{\Q}\qquad\left(\X\in\mathbb{R}^{d}\backslash\bo\right).\label{eq:scalar potential}
\end{equation}
It is harmonic on $\mathbb{R}^{d}\backslash\bo$. 

By definition, the scalar potential is a constant multiple of
the single layer potential of the divergence of $f$,
and it behaves much like the Euclidean gradient of a single layer potential except that
it acts on fields rather than  functions. Note that the integral kernel of $\mathcal{P}$ differs from the integral kernel of $\nabla\SLay\:$ by a minus sign.
As the next lemma shows, it is convenient to define 
the scalar potential in this way.
\begin{lem}
  \label{gradTP}
  For every tangent field $\tau\in\tang$,
the limit
\begin{equation}
\lim_{\conlim}\mathcal{P}\tau\left(\X\right)=\pv\frac{1}{\omega_{d}}\int_{\bo}\frac{\left\langle \Q-\P,\tau(\Q)\right\rangle _{\mathbb{R}^{d}}}{\left|\P-\Q\right|^{d}}\measure{\Q}\doteq\left(\nablat\slay\right)^{\star}\tau\left(\P\right),\label{eq:limit of grad-S-star}
\end{equation}
exists at a.e. $\P\in\bo$, and the limit function is in $\lt$. 
\end{lem}
The operator $\left(\nablat\slay\right)^{\star}\colon\tang\to\lt,$
defined by (\ref{eq:limit of grad-S-star}), is linear and bounded,
and one can see from \eqref{eq:limit of SLP} that $\left(\nablat\slay\right)^{\star}$ is
indeed the $L^{2}$-adjoint of $\nablat\slay$. 

The limits in (\ref{eq:limit of SLP}) and (\ref{eq:limit of grad-S-star})
are independent of the components of the cone $\Gamma(\P)$. That
is, the tangent part of $\nabla\SLay\:$ transitions continuously
through the boundary $\bo$. Contrary to this, the normal component
of $\nabla\SLay\:$ has a jump across the boundary. This behavior is 
will be described
by the double layer potential. 
\begin{rem}
\label{rem:weak divergence}
The operator
$\left(\nablat\slay\right)^{\star}$ is closely connected with the divergence operator: for 
$f\in\tang$, its divergence $\divt f$ is a continuous functional on $\sob$ 
acting on $h$ by the rule $\langle \divt f,h\rangle=-\langle f,\nablat h\rangle$. Now, for  $g\in\lt$, we get that
\[
\left\langle \left(\nablat\slay\right)^{\star}f,g\right\rangle 
=\left\langle f,\nablat\slay g\right\rangle 
=-\left\langle\divt f,  \slay g\right\rangle.
\]
If we identify the dual space of $\lt$ with $\lt$ itself and the dual space of $\sob$ with
the Sobolev space $W^{-1,2}(\bo)$ of negative exponent\footnote{This goes as in the Euclidean case, see \cite[sec. 3.13]{adams03}},
it holds that
$\left(\nablat\slay\right)^{\star}=S^*\divt$ where $S^*\colon W^{-1,2}(\bo)\to L^2(\bo)$. In particular, since $S$ is invertible by Theorem \ref{thm:Inverstable S},
so is $S^*$ and we conclude that $\mathcal{N}\left[\left(\nablat\slay\right)^{\star}\right]$ consists exactly of divergence-free vector fields in $\tang$.
\end{rem}

\paragraph{Double layer potential.}

Let $\eta$ denote the outward unit normal vector on $\bo$. For
$f_{\eta}\in\lt$, the double layer potential of $f_{\eta}$ is
\[
\DLay f_{\eta}\left(\X\right)=\frac{1}{\omega_{d}}\int_{\bo}\frac{\left\langle \Q-\X,\eta(\Q)\right\rangle _{\mathbb{R}^{d}}}{\left|\Q-\X\right|^{d}}f_{\eta}(q)\measure{\Q}=\mathcal{P}(\eta f_{\eta})\left(\X\right),\qquad\left(\X\in\mathbb{R}^{d}\backslash\bo\right)
\]
and the (boundary) double layer potential is
\begin{equation}
\dlay f_{\eta}\left(\P\right)=\pv\frac{1}{\omega_{d}}\int_{\bo}\frac{\left\langle \Q-\P,\eta(q)\right\rangle _{\mathbb{R}^{d}}}{\left|\Q-\P\right|^{d}}f_\eta(q)\measure{\Q},
\end{equation}
defined for a.e. $\P$ on $\bo$. The operator $\dlay\colon\lt \to\lt$
is linear and bounded; we denote its $L^{2}$-adjoint by $\dlay^{\star}$.
These operators qualify the boundary behavior  of $\DLay$ by  the following
Lemma \cite[thms. 1.10 and 1.11]{ver84}.
\begin{lem}
For $f\in\lt$ and for a.e. $\P\in\bo$ it holds that
\begin{align}
\lim_{\conlimi}\DLay f\left(\X\right) & =\left(\frac{1}{2}+\dlay\right)f\left(\P\right),\label{eq:limit for double layer potential}\\
\lim_{\conlimo}\DLay f\left(\X\right) & =-\left(\frac{1}{2}-\dlay\right)f\left(\P\right),\label{eq:limit for the double layer poential outside}
\end{align}
 and 
\begin{align}
\lim_{\conlimi}\left\langle \eta(\P),\nabla\SLay f\left(\X\right)\right\rangle _{\mathbb{R}^{d}} & =-\left(\frac{1}{2}-\dlay^{\star}\right)f\left(\P\right),\label{eq:limit of *double layer potential from inside}\\
\lim_{\conlimo}\left\langle \eta(\P),\nabla\SLay f\left(\X\right)\right\rangle _{\mathbb{R}^{d}} & =\left(\frac{1}{2}+\dlay^{\star}\right)f\left(\P\right).\label{eq:limit for *double layer potential from outside}
\end{align}
\end{lem}
The next two results will be much used in what follows.
\begin{lem}[{\cite[thm. 3.3]{ver84}}]%
\label{lem:invertable double layer Potential}The operators 
\begin{align*}
	\left( \frac{1}{2}+\dlay \right)
		\colon
	\lt & \to\lt & \text{and} &  & 
	\left( \frac{1}{2}-\dlay^{\star} \right) & 
		\colon \ltz\to\ltz,
\end{align*}
are invertible.
\end{lem}
\begin{lem}
\label{lem:S1 constant}There exists a unique function $\consti\in\ltz$,
possibly zero, such that $\slay(1-\consti)$ is a non-zero constant
with
\begin{align}
\left(\frac{1}{2}-\dlay^{\star}\right)(1-\consti) & =0, & \text{and } &  & \nablat\slay(1-\consti) & =0.\label{eq:gradient of SLP}
\end{align}
Further, there exists a constant $C>0$ that depends on $\o$, such
that for every $f\in\lt$ it holds 
\begin{equation} 
\left\|\left(\frac{1}{2}-\dlay^{\star}\right)f\right\|\leq C\:\|\nablat\slay f\|.\label{eq:lower bound for tangent gradient of SLP}
\end{equation}
\end{lem}

The first statement of Lemma \ref{lem:S1 constant} was shown in the proof of \cite[thm. 3.3. (ii)]{ver84}. The second was established in the proof of \cite[thm. 2.1]{ver84}. %
Note  that the equalities $\nablat\slay1=0=\left(\frac{1}{2}-\dlay^{\star}\right)1$ hold if and only if $\consti=0$.  This fact will be used at several places
below.

\begin{rem}\label{rem:null-space-of-K-S}
  Since $\left(\frac{1}{2} - \dlay^{*}\right)$ is invertible on $\ltz$ and maps $L^2(\bo)$ into
  $\ltz$ \cite[thm. 3.3]{ver84}, its null space on $\lt$ is 1-dimensional and consists exactly of multiples of $1-\consti$,
  by \eqref{eq:gradient of SLP}. In another connection, the null space of $\nablat S$ comprises those $f$ for which $S f$ is constant,
  and since $S$ is injective by Theorem  \ref{thm:Inverstable S}
  it follows from \eqref{eq:gradient of SLP} that the null space  of $\nablat S$ coincides with the null space of $\left(\frac{1}{2} - \dlay^{*}\right)$. In fact, the first identity  in \eqref{eq:gradient of SLP} means 
  that $\SLay(1-\nu_0)$ is the solution to a harmonic Neumann problem in $\Omega$ with  nontangential maximal function of its gradient in $L^2(\bo)$ and zero boundary data,  while 
  the second identity says
  that $\SLay(1-\nu_0)$ is the solution to a harmonic Dirichlet problem
  in $\Omega$ with nontangential maximal function in $L^2(\bo)$ and constant boundary data. In both cases, this means that $\SLay (1-\nu_0)$ is constant on $\Omega$.
  \end{rem}

\begin{rem}
Equation \eqref{eq:gradient of SLP} entails that the Newtonian equilibrium measure $\mu$ of
$\overline{\Omega}$---a well studied object in potential theory \cite{wermer74}---is   given by $d\mu = (1/\sigma(\bo)) (1-\consti) \ d\sigma$.  Indeed,
$\mu$ is known to be a probability measure supported on $\bo$, and it is absolutely continuous with $L^2$-density with respect to $\sigma$, by fundamental results on
harmonic measure for Lipschitz domains \cite[cor. to thm. 3]{DahlbergHM1977}
and standard relations between harmonic and equilibrium  measures
\cite[ch. IV, sec. 5, \S 20]{Landkof}. Thus, $d\mu=kd\sigma$ for some
$k\in L^2(\bo)$, and since $\Omega$ is non thin at each point of $\bo$
(for example, by \cite[thm. 5.4 and thm. 5.10]{Landkof},
a characteristic property of
$\mu$ is that its potential is constant on $\Omega\cup\bo$
since it is constant approximately everywhere \cite[ch. II, sec. 1, \S 3]{Landkof}.
This amounts to require that $\slay k$ is constant on $\bo$,
 that is $\nablat\slay k=0$, and by Remark \ref{rem:null-space-of-K-S}
this is  equivalent to
require that $(\frac{1}{2}-K^*)k=0$. Altogether, we conclude
that $1-\nu_0$, which has mass $\sigma(\bo)$,
is equal to $\sigma(\bo)k$, as announced.
We also obtain the extra-information that $\nu_0\leq 1$ a.e. on $\bo$, because $\mu$ is positive. Gruber's conjecture asserts that $\consti$ is zero if and only if $\bo$ is a sphere; and this is known to hold when $\o$ is convex \cite[thm. 4.12]{mitrea09,}. Even though a discussion of Gruber's conjecture is beyond the scope of the present paper, our analysis will stress a link between the equilibrium measure and  the properties of projectors onto Hardy-spaces.
\end{rem}

\section{Orthogonal decomposition of fields}\label{sec:orthoproj}

In this section, we present two orthogonal decompositions of
the space $\ltd d$ based on the null space of the scalar potential.
Hereafter, we will denote the space of harmonic functions in a region $N\subset\mathbb{R}^{d}$ by $\mathcal{H}(N)$. 

We begin with a definition of inner and outer scalar potentials.
\begin{defn}\label{def:potops}
Let $g\vert_{\o}$ denote the restriction of the function $g$ to
the region $\o$, and recall from (\ref{eq:scalar potential}) the scalar
potential $\pot$.
We define the inner scalar potential as
\begin{align}
\poti \colon \ltd d	&	\to \mathcal{H}(\o) \\
				f 	& 	\mapsto(\mathcal{P}f)\vert_{\o}
\end{align}
and the outer scalar potential as
\begin{align}
\poto \colon \ltd d 	&	\to \mathcal{H}(\o^{\exter})\\
				f 	&	\mapsto(\mathcal{P}f)\vert_{\o^{\exter}}.
\end{align}
Both operators are linear.
\end{defn}
The operators $\poti$ and $\poto$ have non-trivial null spaces,
leading to orthogonal direct sum decompositions of $\ltd d$ as $\mathcal{N}\left[\poti\right]\oplus\mathcal{N}\left[\poti\right]^{\perp}$
and $\mathcal{N}\left[\poto\right]\oplus\mathcal{N}\left[\poto\right]^{\perp}$.

When $\bo$ is a sphere, the spaces $\mathcal{N}\left[\poti\right]$ and $\mathcal{N}\left[\poto\right]$ relate to  nontangential limits of harmonic gradients---the Hardy spaces. Below, we recall the definition of Hardy spaces.

\paragraph{Hardy spaces.}

The inner Hardy space, $H_{+}^{2}(\o)$, is the space of  gradients of harmonic functions in $\o$ whose nontangential maximal function is in $\lt$:
\begin{equation}
H^{2}_{+}(\o)=\left\{ \nabla g\,:\,g\in\mathcal{H}(\o),\,\|(\nabla g)^{M}_{\inter}\|<\infty\right\} \label{eq:HardyPlus}.
\end{equation}
The outer Hardy space $H_{-}^{2}(\o)$ is defined similarly on the outside
of $\bo$:
\begin{equation}
H^{2}_{-}(\o)=\left\{ \nabla g\,:\, g\in\mathcal{H}(\o^{\exter}),\,\| (\nabla g)^{M}_{\exter}\|<\infty,\,\lim_{\X\to\infty}\left|g(\X)\right|=0\right\} \label{eq:HardyMinus}.
\end{equation}
For each $\nabla g\in H_{+}^{2}\left(\o\right)$ there is a square integrable
boundary field $f\in\ltd d$, such that $\nabla g$ converges nontangentially
from inside to $f$, a.e. on $\bo$ (see for example \cite[thm. 7.9]{gilmur91}). The space of boundary
fields obtained in this way
is denoted by $\hardyp$.
The same limiting procedure for 
$H^{2}_{-}(\o)$-functions  leads to the space $\hardym$. We still
call the space $\hardyp$ (or $\hardym$) the
inner (or outer) Hardy space, with no fear of confusion because the nontangential limit and the $L^2$-boundedness of the nontangential maximal function characterize harmonic functions.

If $\bo$ is a sphere then $\hardyp$ and $\hardym$ are orthogonal;
and denoting by $\divf\subset\tang$ the space of divergence-free tangent
fields on $\bo$, we have that  $\mathcal{N}\left[\poti\right]=\hardyp\oplus\divf$
and $\mathcal{N}\left[\poti\right]^{\perp}=\hardym$. This is
the Hardy-Hodge decomposition on a sphere---an orthogonal direct sum decomposition
of $\ltd d$ as $\ltd d=\hardyp\oplus\hardym\oplus\divf$. Moreover,
if we use the null space of $\poti$ instead of the null space of $\poto$, 
the role of $\hardyp$ and $\hardym$ get swapped.

On Lipschitz domains, the Hardy-Hodge decomposition still exists, but the inner and outer Hardy-spaces are no longer orthogonal to each other
when $\bo$ is not a sphere (see Corollary \ref{cor:IODortho}).

\subsection{Inner decomposition}\label{sec:inner decomposition}

Below we derive an orthogonal direct sum for $\ltd d$,
based on the null space of the inner scalar potential. For $f$  a field, recall
its normal component $f_{\eta}$ and its tangent component $f_{\tcomp}$ such that $f=\eta f_{\eta} + f_{\tcomp}$.
\begin{defn}\label{def:ibo}
Define the operator, $\tri\colon\ltd d\to\lt $, as, 
\begin{equation}
\tri f \doteq\left(\frac{1}{2}+\dlay\right)f_{\eta}+\left(\nablat\slay\right)^{\star}f_{\tcomp};\label{eq:definition of B}
\end{equation}
it is linear and bounded. We call $\tri$ the inner boundary
operator.
\end{defn}
\begin{thm}\label{thm:npi}
\label{thm:The-inner-potenential}\label{thm:null space for PI}The inner scalar potential of $f\in\ltd d$ can be expressed as
\begin{equation}
\poti f\left(\X\right)=\DLay\left(\frac{1}{2}+\dlay\right)^{-1}\tri f\left(\X\right)\qquad\left(\X\in\o\right).\label{eq:inner potential with B}
\end{equation}
Moreover, the null space of $\poti$ is given by
\begin{equation}
\mathcal{N}\left[\poti\right]=\mathcal{N}\left[\tri\right]=\left\{ f\in\ltd d\,:\left(\frac{1}{2}+\dlay\right)f_{\eta}=-\left(\nablat\slay\right)^{\star}f_{\tcomp}\right\}.\label{eq:null space of inner potential}
\end{equation}
\end{thm}
\begin{proof}
Let $f$ be in $\ltd d$. For
$\X\in\o$,  we get that
\begin{align}
\poti f\left(\X\right) & =\poti(\eta f_{\eta})\left(\X\right)+\poti f_{\tcomp}\left(\X\right)=\DLay f_{\eta}\left(\X\right)+\poti f_{\tcomp}\left(\X\right).
\end{align}
Taking the nontangential limit on $\bo$ from $\Omega$ on both sides of
the above equation and using (\ref{eq:limit of grad-S-star}) and
(\ref{eq:limit for double layer potential}) yields
\begin{equation}
\lim_{\X\to\Q}\poti f\left(\X\right)=\left(\frac{1}{2}+\dlay\right)f_{\eta}\left(\Q\right)+\left(\nablat\slay\right)^{\star}f_{\tcomp}\left(\Q\right)=\tri f\left(\Q\right).\label{eq:limit of the inner potential}
\end{equation}
By Lemma \ref{lem:invertable double layer Potential} the operator
$\left(\frac{1}{2}+K\right)^{-1}$ is well defined. Hence,  $\poti f-\DLay\left(\frac{1}{2}+\dlay\right)^{-1}\tri f$
is harmonic in $\Omega$, and converges nontangentially to zero a.e. on
$\bo$, and has $L^2$-bounded nontangential maximal function. By uniqueness of the Dirichlet problem  (see for example \cite[cor. 3.2]{ver84}) that function
is zero; this proves \eqref{eq:inner potential with B}. It follows that $\mathcal{N}\left[\poti\right]=\mathcal{N}\left[\tri\right]$
and the second equality in \eqref{eq:null space of inner potential}
is immediate from the definition of $\tri$. 
\end{proof}
\begin{cor}
  \label{cor:null space of PI (normal component)}For $f\in \mathcal{N}\left[\poti\right]$, the normal component of $f$ is uniquely determined by the tangent
  part of $f$ through
the relation, 
\begin{equation}
f_{\eta}=\left(\frac{1}{2}+\dlay\right)^{-1}\left(\nablat\slay\right)^{\star}f_{\tcomp}.
\end{equation}
\end{cor}
Theorem \ref{thm:npi} shows that the operator $\tri$ determines the null space for $\poti$. The following lemma shows that the adjoint $\tric$ of $\tri$
determines the orthogonal complement of the null space of $\poti$. 
\begin{lem}
\label{lem:adjoint of BI}The
operator $\tric\colon\lt\to\ltd d$ can be written as 
\begin{equation}
\tric g=\eta\left(\frac{1}{2}+\dlay^{\star}\right)g+\nablat\slay g.
\end{equation}
The range of $\tric$ satisfies $\mathcal{R}\left[\tric\right]=\hardym=\mathcal{N}\left[\poti\right]^{\perp}$.
\end{lem}
\begin{proof}
Let $f=\eta f_{\eta} + f_{\tcomp}$ be in $\ltd d$. For every
$g\in\lt$, we get that
\begin{align}
\left\langle \tri f,g\right\rangle  & =\left\langle \left(\frac{1}{2}+\dlay\right)f_{\eta}+\left(\nablat\slay\right)^{\star}f_{T},g\right\rangle \\
 & =\left\langle f_{\eta},\left(\frac{1}{2}+\dlay^{\star}\right)g\right\rangle +\left\langle f_{T},\nablat\slay g\vphantom{\left(\frac{1}{2}\right)}\right\rangle =\left\langle f,\tric g\right\rangle.
\end{align}
Hence, using (\ref{eq:limit for *double layer potential from outside})
we have that
\begin{align}
\tric g\left(\P\right) & =\eta\left(\frac{1}{2}+\dlay^{\star}\right)g\left(\P\right)+\nablat\slay g\left(\P\right)=\lim_{\X\to\P}\nabla\SLay g\left(\X\right)\label{eq:Bi-star}
\end{align}
when $\X$ approaches $\P$ within $\Gamma_{\exter}$. Since $\SLay g$ is harmonic in $\o^{\exter}$ and has $L^2$-bounded maximal function, the right hand side of \eqref{eq:Bi-star} defines a function in $\hardym$.

To see that $\tric$ is onto, consider  $f\in \hardym$. By  the well-posedness
of the Neumann problem, there exists a unique function $\varphi$ that is harmonic in $\o^{\exter}$, vanishes at infinity, and  whose normal derivative equals $f_{\eta}$ a.e. on $\bo$, while  $\| (\nabla\varphi)^{M} \|$ is finite.
On the other hand, since $\left(\frac{1}{2} + \dlay^{\star}\right)$ is invertible because its adjoint is by Lemma \ref{lem:invertable double layer Potential},
we get from\eqref{eq:limit for *double layer potential from outside} that $\varphi = \SLay \left(\frac{1}{2} + \dlay^{\star}\right)^{-1} f_{\eta}$, and so the function $g=\left(\frac{1}{2} + \dlay^{\star}\right)^{-1} f_{\eta}\in\lt$ is such that 
\[
\displaystyle f(\P)=\lim_{\X \to \P} \nabla \varphi \left(\P\right) = \lim_{\X \to \P} \nabla \SLay g \left(\X\right)
\]
where the nontangential limits are taken from $\Omega^o$. Consequently, $\tric$ is onto and $\mathcal{R}\left[\tric\right]=\hardym$.

For the last statement, observe by general properties of operators on Hilbert space that
$\mathcal{N}\left[\tri\right]=\mathcal{R}\left[\tric\right]^{\perp}=\hardym^{\perp}$.  The conclusion now follows since $\hardym$ is closed.
\end{proof}
\begin{cor}
The space $\ltd d$ splits into an orthogonal sum as 
\begin{equation}
\ltd d=
\hardym\oplus\mathcal{N}\left[\poti\right]. \label{eq:HmNPi}
\end{equation}
\end{cor}

\subsection{Orthogonal projections for the inner decomposition}\label{sec:innerproj}

For the rest of the paper, we make the following definitions.
\begin{defn}
\label{def:tangent spaces}We define a space of tangent fields with
zero mean as
\begin{align*}
\tangz & =\left\{ f\in\tang\,:\,\left(\nablat\slay\right)^{\star}f\in\ltz\right\} ,
\end{align*}
and the space of tangent gradient fields as 
\begin{align*}
\grad & =\left\{ \nablat\varphi\,:\, \varphi\in\sob\right\} ,
\end{align*}
as well as the space of tangent divergence-free vector fields
(see Remark \ref{rem:weak divergence}):
\begin{align*}
\divf & =\left\{ f\in\tang\,:\,\left(\nablat\slay\right)^{\star}f=0\right\}.
\end{align*}

Since $\slay$ is onto $\sob$, we can equivalently write the space of gradient fields as $\grad  =\left\{ \nablat\slay f\,:\,f\in\lt\right\}$.
Then, it is immediate that $\divf$ is the orthogonal complement of $\grad$;
this fact is also apparent from the definition of the divergence. The spaces $\tangz$ and $\divf$ are closed, since $\left(\nablat\slay\right)^{\star}$
is continuous, and therefore there exist orthogonal projections $\projz \colon \tang\to\tangz$
and $\projd:\tang\to\divf$.
\end{defn}
The goal of this section is to prove the following theorem.
\begin{thm}
\label{thm:L2 decomposition in}The operator 
\begin{equation}
\projm=\tric\left(\tri\tric\right)^{-1}\tri
\end{equation}
defines the  orthogonal projection from $\ltd d$ onto $\hardym$.

Let $\projd$ be as in definition \ref{def:tangent spaces}. Then,
\begin{equation}
\projv\doteq\mathds{1}-\projd-\projm
\end{equation}
is the orthogonal projection from $\ltd d$ onto $\Iod=\mathcal{N}[\poti]\cap\mathcal{N}[\poto]^{\perp}$; and the space $\ltd d$ splits into an orthogonal direct sum as 
\begin{equation}
\ltd d=\Iod \oplus \hardym \oplus \divf. \label{eq:HmID}
\end{equation}
\end{thm}
We will proceed in several steps. First, we derive an explicit expression
for the projection $\projd$ that leads to an explicit formula, in terms of potentials, 
for the Helmholtz-Hodge-decomposition of tangent fields on Lipschitz surfaces.
Next, we show that the operator $\tri\tric$ is invertible, thereby implying
that
$\projm$ introduced  in the theorem is a well defined orthogonal
projection. We conclude the proof by showing that the ranges of $\projm$ and $\projv$ are as stated. 

In what follows, we let $\restri$ denote the canonical projection of
$\lt$ onto $\ltz$.
\begin{lem}
\label{lem:grad SLP invertable}The operator $\restri\left(\nablat\slay\right)^{\star}\nablat\slay \colon \ltz\to\ltz$
is self-adjoint and invertible.
\end{lem}
\begin{proof}
The operator $\restri\left(\nablat\slay\right)^{\star}\nablat\slay$ is self-adjoint, since for $f,g\in\ltz$ we have: 
\begin{align}
\left\langle \restri\left(\nablat\slay\right)^{\star}\nablat\slay f,g\right\rangle  & =\left\langle \left(\nablat\slay\right)^{\star}\nablat\slay f,g\right\rangle \nonumber \\
 & =\left\langle f,\left(\nablat\slay\right)^{\star}\nablat\slay g\right\rangle =\left\langle f,\restri\left(\nablat\slay\right)^{\star}\nablat\slay g\right\rangle.
\end{align}
By the Riesz lemma it defines a symmetric and continuous bilinear form,
say $M$, such that for $f,g\in\ltz$ 
\begin{equation}
M\left(f,g\right)\doteq\left\langle \restri\left(\nablat\slay\right)^{\star}\nablat\slay f,g\right\rangle.
\end{equation}
From (\ref{eq:lower bound for tangent gradient of SLP}) and Lemma
\ref{lem:invertable double layer Potential}, it follows
that $M$ is coercive, and by the Lax-Milgram theorem $\restri\left(\nablat\slay\right)^{\star}\nablat\slay$
has bounded inverse.
\end{proof}
\begin{lem}
\label{thm:Hodge decomposition on Lipschitz boundary}Let $\projz$
and $\projd$ be the projections from definition \ref{def:tangent spaces}.
Then, 
\begin{align}
\projd & =\projz\left(\mathds{1}-\nablat S\left(\restri\left(\nablat\slay\right)^{\star}\nablat\slay\right)^{-1}\left(\nablat\slay\right)^{\star}\right)\projz.\label{eq:projection D}
\end{align}
\end{lem}
\begin{proof}
Let $\consti$ be as in Lemma \ref{lem:S1 constant}.
The projection onto $\tangz$ can be written as
\begin{align}
\projz f & =\left(\mathds{1}-\projs\right)f, 
& \text{with} \qquad
\projs f & =\begin{cases}
\left\langle f,\nablat\slay1\right\rangle \frac{\nablat\slay1}{\|\nablat\slay1\|^{2}} & \text{if }\,\consti\neq0,\\
0 & \text{if }\,\consti=0.
\end{cases}
\end{align}
To see this, observe that $\projz$ and $\projs$ as defined above are complementary  projections. Then, for $f\in\tang$, we have that
\begin{align*}
\int_{\bo} \left(\nablat\slay\right)^{\star} \projz f 
=\int_{\bo}\left(\nablat\slay\right)^{\star}(\mathds{1}-\projs)f\left(\Q\right)\measure{\Q}%
= \left\langle f, \nablat\slay 1\right\rangle - \left\langle f , \nablat\slay 1\right\rangle
=0,
\end{align*}
and thus $\projz f $ is in $\tangz$.
Conversely, if $f$ is in $\tangz$ then $\projs f =0$ and $\projz f = f$. Hence, $\projz\tang=\tangz$. 

Next, we write the operator $\projd$ from (\ref{eq:projection D})
as 
\begin{align}
\projd & =\projz\left(\mathds{1}-A\right)\projz\qquad\text{with }\qquad A=\nablat\slay\left(\restri\left(\nablat\slay\right)^{\star}\nablat\slay\right)^{-1}\left(\nablat\slay\right)^{\star},\label{eq:projection PD}
\end{align}
and we show that it is the orthogonal projection onto $\divf$. 

The operator $\projz A\projz$ is well defined, because $\projz$
projects $\tang$ onto $\tangz$ and $\left(\nablat\slay\right)^{\star}$
maps the latter into $\ltz$, where the operator $\left(\restri\left(\nablat\slay\right)^{\star}\nablat\slay\right)^{-1}$
is well defined by Lemma \ref{lem:grad SLP invertable}. Also, the
operator $A$ preserves the space $\tangz$; that is, for $f\in\tangz$,
we have $Af\in\tangz$. Indeed, observe that if $f$ is in $\tangz$, then by (\ref{eq:gradient of SLP})  and Lemma \ref{lem:grad SLP invertable} we have:
\begin{multline*}
\left\langle\left(\nablat\slay\right)^{\star}Af,1\right\rangle
=\left\langle Af,\nablat\slay\consti\right\rangle 
=\left\langle (R_0(\nablat\slay)^*\nablat\slay)^{-1} (\nablat\slay)^* f,
(\nablat\slay)^*\nablat\slay \nu_0\right\rangle \\
=\left\langle (R_0(\nablat\slay)^*\nablat\slay)^{-1} (\nablat\slay)^* f,
R_0(\nablat\slay)^*\nablat\slay \nu_0\right\rangle
=
\left\langle f,\nablat\slay\consti\right\rangle =\left\langle f,\nablat\slay1\right\rangle =0.
\end{multline*}
Consequently, $Af$ is in $\tangz$. By a similar calculation, $AA\projz=A\projz$ and thus  $A\projz A\projz=AA\projz=A\projz$.

From the above we see that $\projd$ is idempotent, since
\begin{equation}
\projd\projd=\projz\left(\mathds{1}-A\right)\projz\left(\mathds{1}-A\right)\projz=\projz\left(\mathds{1}-A\right)\projz=\projd.
\end{equation}
Also, $\projd$ is self-adjoint, because $\mathds{1},\,\projz,$ and
$A$ are self-adjoint. Thus, $\projd$ is an orthogonal projection. 

We claim that the range of $\projd$ is $\divf$. To support this
claim, it is enough to show that $f\in\divf$ holds if and only if
$\projd f=f$. Assume that $\projd f=f$ holds. Since the range of $\projd$ is in
the range of $\projz$, it follows that $\projz f=f\in\tangz$, and
since $A$ preserves $\tangz$ we get from (\ref{eq:projection PD}) that
\begin{equation}
Af=\projz A\projz f=\projz f-\projd f=f-f=0.
\end{equation}
This implies $\left(\nablat\slay\right)^{\star}f=0$, as the remaining
operators in the definition of $A$ are injective on functions with zero mean, by Lemma \ref{lem:S1 constant}. Therefore, $f$ is in
$\divf$.

Conversely, take  $f\in\divf$ and write, $f=c\nablat\slay1+\projz f\left(=\projs f+\projz f\right)$
for some constant $c\in\mathbb{R}$. Then, 
\begin{equation}
0=\left\langle \left(\nablat\slay\right)^{\star}f,1\right\rangle =c\left\langle \nablat\slay1,\nablat\slay1\right\rangle +\left\langle \projz f,\nablat\slay1\right\rangle =c\|\nablat\slay1\|^{2}.
\end{equation}
Hence, either $\nablat\slay1 =0$ or $\nablat\slay1\neq0$, and in the latter case  $c=0$; thus, $\projz f=f$ in all cases, so that
$f$ is in $\tangz$ and hence, 
\begin{align*}
A\projz f=\nablat\slay\left(\restri\left(\nablat\slay\right)^{\star}\nablat\slay\right)^{-1}\left(\nablat\slay\right)^{\star}f & =0,
\end{align*}
ensuing that $\projd f=\projz f + \projz A \projz f=f$. This shows that the range of $\projd$
is $\divf$ and completes the proof.
\end{proof}

Since divergence-free fields are orthogonal to gradients, it follows that
\begin{equation}
\projg\doteq\mathds{1}-\projd=\projs+\projz A\projz,\label{eq:projection PG}
\end{equation}
is the orthogonal projection  onto $\grad$, so we can write the Helmholtz-Hodge
decomposition explicitly as
\begin{equation}
  \label{HHdec2}
\tang=\projg\tang+\projd\tang=\grad+\divf.
\end{equation}
\begin{rem}
  \label{HT}
  The existence of the Helmholtz-Hodge decomposition is trivial, for it reduces to the fact that a Hilbert space decomposes as the sum of a closed subspace and
  its orthogonal complement. However, the point in \eqref{HHdec2} is the explicit expression in terms of
  layer potentials.
  Note that the divergence-free term can split further by Hodge theory on Lipschitz orientable Riemannian manifolds
  \cite{teleman1983index}.
\end{rem}
\begin{rem}
\label{rem:tangent projections on fields}
The definition of the projections $\projd$
and $\projg$ can be extended to the whole of $\ltd d$ as follows: for $f=\eta f_{\eta} + f_{\tcomp}\in\ltd d$ put
\begin{align}
\projD f & =\projd f_{\tcomp}, & \projDp f & =\eta f_{\eta}+\projg f_{\tcomp}.
\end{align}
\end{rem}

\begin{cor}
\label{cor:tangent part deteremind by the normal}Let $f=\eta f_{\eta}+f_{\tcomp}$
be in $\mathcal{N}\left[\poti\right]$.
Then, $f_{\eta}$ determines $f_{\tcomp}$ uniquely up to a divergence-free tangent field. 
\end{cor}

\begin{proof}
For $f\in\mathcal{N}\left[\poti\right]$, the relation between $f_{\eta}$ and $f_{\tcomp}$ from  Theorem \ref{thm:null space for PI} reads as
\begin{equation}
\left(\frac{1}{2}+\dlay\right)f_{\eta}=-\left(\nablat\slay\right)^{\star} f_{\tcomp}.\label{eq:relation_between_components}
\end{equation}
The divergence-free component of $f_{\tcomp}$ does not contribute to this relation as it satisfies $\left(\nablat\slay\right)^{\star}\projd f_{\tcomp}=0$. Therefore, we can assume without loss of generality that $f_{\tcomp}$ is a gradient field. To prove the corollary, we have to invert the right hand side of \eqref{eq:relation_between_components}. The difficulty is that the left hand side of
\eqref{eq:relation_between_components} may not have zero mean and thus, we cannot directly use Lemma \ref{lem:grad SLP invertable}.

To circumvent this, we split $f_{\tcomp}$ into a field with zero mean plus
$\nablat\slay c$ for some constant $c$. Specifically, since $f_{\tcomp}$ is a gradient field, there exists a constant $c$ and a function $\varphi\in\ltz$ such that $f_{\tcomp}=\nablat\slay\left(c + \varphi \right)$, and then $\nablat\slay\varphi$ is in $\tangz$ (this is simply the decomposition $f_{\tcomp} = \projg f_{\tcomp} = \projs f_{\tcomp} + \projz A \projz f_{\tcomp}$ where we set $\nablat\slay\varphi = \projz A\projz f_{\tcomp} $, which is possible because both $f_{\tcomp}$ and $\projs f_{\tcomp}$ are gradients). Inserting this decomposition into \eqref{eq:relation_between_components} and rearranging terms, we get that
\begin{equation}
\left(\nablat\slay\right)^{\star}\nablat\slay\varphi=-\left(\frac{1}{2}+\dlay\right)f_{\eta}-c\left(\nablat\slay\right)^{\star}\nablat\slay1.
\end{equation}
Since the left hand side is in $\ltz$, the right hand side also has zero
mean. Therefore, either $\nablat\slay1=0$ and then Lemma
\ref{lem:grad SLP invertable} achieves the proof, or else
\begin{align}
c=-\frac{1}{\|\nablat\slay 1 \|^{2}} \left\langle\left(\frac{1}{2}+K\right) f_{\eta},1\right\rangle.
\end{align}
In the latter case, applying
$(\restri\left(\nablat\slay\right)^{\star}\nablat\slay)^{-1}$
gives us
\begin{equation}
\varphi=\left(\restri\left(\nablat\slay\right)^{\star}\nablat\slay\right)^{-1}\left[\left\langle\left(\frac{1}{2}+K\right)f_{\eta},1\right\rangle\frac{\left(\nablat\slay\right)^{\star}\nablat\slay1}{\|\nablat\slay 1\|^{2}}-\left(\frac{1}{2}+\dlay\right)f_{\eta}\right],\label{eq:tangent from normal}
\end{equation}
whence the constant $c$ and the function $\varphi$ are both determined by $f_{\eta}$. Thus, so is the tangent field $f_{\tcomp}$, as desired.
\end{proof}

We turn to the inverse of $\tri\tric$.
\begin{lem}
\label{lem:properties of B-star}Let $\tri$ be the inner boundary
operator from definition \ref{def:ibo} and let $\tric$
denote its adjoint from Lemma \ref{lem:adjoint of BI}, then $\tri\tric \colon \lt\to\lt$
is invertible.
\end{lem}
\begin{proof}
Define a symmetric continuous bilinear form, $M$, such that, 
\begin{align}
M\left(f,g\right) & \doteq\left\langle \tri\tric f,g\right\rangle \qquad\left(f,g\in\lt\right).
\end{align}
Since $\left(\frac{1}{2}+\dlay^{\star}\right)$ is invertible, there
exists a positive constant $C$ such that, 
\begin{align}
M\left(f,f\right) & =\|\tric f\|^{2}\geq\left\|\left(\frac{1}{2}+\dlay^{\star}\right)f\right\|^{2}\geq C\|f\|^{2}.
\end{align}
Thus, $M$ is coercive and $\tri\tric$ is invertible by the Lax-Milgram
theorem.
\end{proof}
The proof of the Theorem \ref{thm:L2 decomposition in} is now an
easy consequence of the previous results. 
\begin{proof}[Proof of Theorem \ref{thm:L2 decomposition in}]
Let $\projd$ be as above, and recall the  operators introduced in
the theorem: $\projm=\tric\left(\tri\tric\right)^{-1}\tri$ and $\projv=\mathds{1}-\projd-\projm$.
We show the following: i) $\projm$ is the  orthogonal projection onto $\hardym$; ii) the range of $\projd$ is 
$\left\{ f\in\mathcal{N}\left[\poti\right]\,:\,f_{\eta}=0\right\} $;
iii) $\projv$ is the orthogonal projection onto $\Iod$.

i): that $\projm$ is an orthogonal projection is readily verified,
since it is self-adjoint and idempotent. From Corollary \ref{thm:The-inner-potenential} 
we have that $\mathcal{N}\left[\poti\right]=\mathcal{N}\left[\tri\right]=\mathcal{N}\left[\projm\right]$,
and from Lemma \ref{lem:adjoint of BI} we get  that $\mathcal{R}\left[\projm\right]=\mathcal{N}\left[\poti\right]^{\perp}=\hardym$.
This proves i).

ii): by Lemma \ref{thm:Hodge decomposition on Lipschitz boundary},
$\projd$ projects onto $\divf$. The rest of the assertion follows
from \eqref{eq:null space of inner potential} and the injectivity of $\left(\frac{1}{2}+\dlay\right)$.

iii): we show the projection property first. The operator $\projv$
is self-adjoint because $\mathds{1},\,\projd,$ and $\projm$
are self-adjoint. For the idempotence we compute, 
\begin{equation}
\projv\projv=\mathds{1}-\projd+\projm-\projm\left(\mathds{1}-\projd\right)-\left(\mathds{1}-\projd\right)\projm.\label{eq:idempotence of P_vert}
\end{equation}
By i) and ii) the space $\hardym$ is orthogonal to $\divf$, so that $\projm\projd=\projd\projm=0$ and
\begin{align}
\projm\left(\mathds{1}-\projd\right) & =\projm, & \left(\mathds{1}-\projd\right)\projm & =\projm.
\end{align}
Using these identities, the right hand side of (\ref{eq:idempotence of P_vert})
becomes $\mathds{1}-\projd-\projm=\projv$.

From the definition of $\projv$, it is clear that $\projv\ltd d$ is the orthogonal complement of $\divf$ in $\mathcal{N}[\poti]$.
Further, if $f$ lies in $\projv \ltd d$ and is silent outside as well,
then $f\in\mathcal{N}[\poti]\cap\mathcal{N}[\poto]$. That $\mathcal{N}[\poti]\cap\mathcal{N}[\poto]=\divf$
is shown in Lemma \ref{lem:exclusively_silent}, further below. Assuming this result for now, it follows that the only field that is in $\projv\ltd d$ and in the intersection of the null spaces is the zero field. This proves the assertion.

Finally, we have that $\projm+\projv+\projd=\mathds{1},$ and thus,
\begin{equation}
\ltd d=\projv\ltd d \oplus \projm\ltd d\oplus \projd\ltd d.
\end{equation}
This completes the proof.
\end{proof}

\subsection{Outer decomposition}\label{sec:outer decomposition}

In this section we derive an orthogonal decomposition of $\ltd d$
based on the null space of the outer scalar potential. 
\begin{defn}
Define the operator, $\tro \colon \ltd d\to\lt $, as
\begin{equation}
\tro f\doteq-\left(\frac{1}{2}-\dlay\right)f_{\eta}+\left(\nablat\slay\right)^{\star}f_{\tcomp};\label{eq:definition of B-1}
\end{equation}
it is linear and bounded. We call $\tro$ the outer boundary
operator.
\end{defn}
As in the previous section, we can write the outer potential as a
harmonic extension of the outer boundary operator. Contrary to that
section, however, we need to take separate care of the constant
mode, because the operators $\frac{1}{2}-\dlay$ and $\frac{1}{2}-\dlay^{\star}$
are invertible on $\ltz$ only.
\begin{thm}\label{cor:npo}
\label{thm:The-outer-potenential}\label{thm:null space for PO}Let
$\consti\in\ltz$ denote the function from Lemma \ref{lem:S1 constant}.
The outer potential of $f\in\ltd d$ can be expressed as
\begin{multline}
\poto f\left(\X\right)=\mathcal{-\DLay}\left(\frac{1}{2}-\dlay\right)^{-1}\left(\tro f-\left\langle \tro f,1\right\rangle \right)\left(\X\right)\\
+\left\langle \tro f,1\right\rangle \frac{\SLay\left(1-\consti\right)\left(\X\right)}{\left|\slay\left(1-\consti\right)\right|},\qquad\left(\X\in\o^{\exter}\right).\label{eq:outer potential with B}
\end{multline}
The null space of $\poti$ is given by
\begin{equation}
\mathcal{N}\left[\poto\right]=\mathcal{N}\left[\tro\right]=\left\{ f\in\ltd d\,:\left(\frac{1}{2}-\dlay\right)f_{\eta}=\left(\nablat\slay\right)^{\star}f_{\tcomp}\right\}.\label{eq:null space of outer potential}
\end{equation}
\end{thm}
\begin{proof}
From (\ref{eq:gradient of SLP}) we know that $\slay(1-\consti)$
is a non-zero constant, hence the second term in (\ref{eq:outer potential with B})
is well defined and 
converges to $\left\langle \tro f,1\right\rangle $ a.e. as $\X$ tends to
the boundary nontangentially, by (\ref{eq:boundary single layer potential}). Consequently, the right hand side of (\ref{eq:outer potential with B})
converges to $\tro f$ under the same conditions. Taking nontangential limits
on $\bo$ from $\Omega^o$ on
both sides of (\ref{eq:outer potential with B}) while  using (\ref{eq:limit of grad-S-star})
and (\ref{eq:limit for the double layer poential outside}), we get that
\begin{equation}
\lim_{\X\to\Q}\poto f\left(\X\right)=-\left(\frac{1}{2}-\dlay\right)f_{\eta}\left(\Q\right)+\left(\nablat\slay\right)^{\star}f_{\tcomp}\left(\Q\right)=\tro f\left(\Q\right).\label{eq:limit of the outer potential}
\end{equation}
It follows that both hand sides of (\ref{eq:outer potential with B}) are
harmonic functions on $\o^{\exter}$ with $L^2$-bounded nontangential maximal function,
whose nontangential limits coincide a.e. on $\bo$. So, by  uniqueness of a solution
to the Dirichlet problem,
they are equal. It follows that $\mathcal{N}\left[\poto\right]=\mathcal{N}\left[\tro\right]$; the second equality in (\ref{eq:null space of outer potential})
is immediate from the definition of $\tro$. 
\end{proof}
\begin{cor}
Let $f=\eta f_{\eta}+f_{\tcomp}$
be in $\mathcal{N}\left[\poto\right]$.
Then, $f_{\eta}$ determines $f_{\tcomp}$ uniquely, up to a divergence-free tangent field. 
\end{cor}
\begin{proof}
The proof follows the same steps as the one  of Corollary \ref{cor:tangent part deteremind by the normal}.
\end{proof}
Like in the previous section, the $L^{2}$-adjoint of $\tro$, denoted as $\troc$, helps one to characterize the orthogonal complement of the null space for $\poto$: 
\begin{lem} \label{lem:adjoint of BO}
The operator $\troc \colon \lt\to\ltd d$ can be expressed as
\begin{equation}
\troc g=-\eta\left(\frac{1}{2}-\dlay^{\star}\right)g+\nablat\slay g.
\end{equation}
The range of $\troc$ satisfies $\mathcal{R}\left[\troc\right]=\hardyp=\mathcal{N}\left[\poto\right]^{\perp}$.
\end{lem}
\begin{proof}
Let $f$ be in $\ltd d$ with normal component  $f_{\eta}$, and tangential part $f_{\tcomp}$. Then, for
every $g\in\lt$ we get that
\begin{align}
\left\langle \tro f,g\right\rangle  & =\left\langle -\left(\frac{1}{2}-\dlay\right)f_{\eta}+\left(\nablat\slay\right)^{\star}f_{\tcomp},g\right\rangle \\
 & =\left\langle f_{\eta},-\left(\frac{1}{2}-\dlay^{\star}\right)g\right\rangle +\left\langle f_{\tcomp},\nablat\slay g\vphantom{\left(\frac{1}{2}\right)}\right\rangle =\left\langle f,\troc g\right\rangle.
\end{align}
Hence, using \eqref{eq:limit of *double layer potential from inside}
we obtain:
\begin{align}
\troc g\left(\P\right) & =-\eta\left(\frac{1}{2}-\dlay^{\star}\right)g\left(\P\right)+\nablat\slay g\left(\P\right)=\lim_{\X\to\P}\nabla\SLay g\left(\X\right), \label{eq:Bo-star}
\end{align}
where the limit is taken from inside $\Omega$. The rest of the proof follows the same steps as in Lemma \ref{lem:adjoint of BI}; the only difference is that $\left(\frac{1}{2}-\dlay^{\star}\right)$ is merely invertible on functions with zero mean and thus, we can only define $\varphi = \SLay \left(\frac{1}{2}-\dlay^{\star}\right)f_{\eta}$ if $f_{\eta}$  is in $\ltz$. This last condition, however, is guaranteed by the Gauss theorem when $f_{\eta}$ is the normal component of a
harmonic gradient in $\o$.

\end{proof}
\begin{cor}
The space $\ltd d$ splits into an orthogonal sum as 
\begin{equation}
\ltd d
=\hardyp\oplus\mathcal{N}\left[\poto\right].\label{eq:HpNPo}
\end{equation}
\end{cor}

\subsection{Orthogonal projections for the outer decomposition}\label{sec:outerproj}

Let $\consti$ denote the function from Lemma \ref{lem:S1 constant}.
Define two complementary projections, 
\begin{align}
\projtp f & =\left(\mathds{1}-\projt\right)f, & \projt f & =\begin{cases}
\left\langle f,\troc1\right\rangle \frac{\troc1}{\|\troc1\|^{2}} & \text{if }\,\consti\neq0,\\
0 & \text{if }\,\consti=0,
\end{cases}
\end{align}
and recall that $\restri$ denotes the projection of $\lt$ to $\ltz$.
In this section we prove the following theorem.
\begin{thm}
\label{thm:L2 decomposition out}The operator 
\begin{equation}
\projp=\projt+\projtp\troc\left(\restri\tro\troc\right)^{-1}\tro\projtp\label{eq:P-plus}
\end{equation}
defines an orthogonal projection from $\ltd d$ onto $\hardyp.$

For $\projd$ as in definition \ref{def:tangent spaces}, the operator
\begin{align}
\projc & \doteq\mathds{1}-\projd-\projp
         \label{projO}
\end{align}
is the orthogonal projection from $\ltd d$ onto 
$\iOd=\mathcal{N}[\poto]\cap\mathcal{N}[\poti]^{\perp}$. Moreover,
the space $\ltd d$ splits into an orthogonal direct sum as 
\begin{equation}
\ltd d=\hardyp \oplus \iOd \oplus \divf. \label{eq:HpOD}
\end{equation}
\end{thm}
Compared with the previous section, the only additional ingredient
needed for the proof of this theorem is the fact that the
operator $\restri\tro\troc$ is invertible. We show this in the following
lemma. 
\begin{lem}
\label{lem:RBB operator}$\restri\tro\troc \colon \ltz\to\ltz$ is self-adjoint
and invertible. 
\end{lem}
\begin{proof}
The operator $\restri\tro\troc$ is self-adjoint, for if $f,g\in\ltz$
we have
\begin{equation}
\left\langle \restri\tro\troc f,g\right\rangle =\left\langle \tro\troc f,g\right\rangle =\left\langle f,\tro\troc g\right\rangle =\left\langle f,\restri\tro\troc g\right\rangle.
\end{equation}
Hence, it defines a symmetric bounded functional, $M$, such that
for $f,g\in\ltz$,
\begin{equation}
M\left(f,g\right)=\left\langle \restri\tro\troc f,g\right\rangle.
\end{equation}
Since $\frac{1}{2}-\dlay^{\star}$ is invertible on $\ltz$ by Lemma \ref{lem:invertable double layer Potential}, there
exists a constant $C$ such that
\begin{equation}
M\left(f,f\right)=\|\troc f\|^{2}\geq\left\|\left(\frac{1}{2}-\dlay^{\star}\right)f\right\|^{2}\geq C\|f\|^{2}.
\end{equation}
By the Lax-Milgram theorem, $\restri\tro\troc$ has a bounded inverse. 
\end{proof}
\begin{proof}[Proof of Theorem \ref{thm:L2 decomposition out}]
With $\projd$ as above, write the operator $\projp$ from
(\ref{eq:P-plus}) as
\begin{align*}
\projp & =\projt+\projtp C\projtp, & \text{with }\quad C & =\troc\left(\restri\tro\troc\right)^{-1}\tro,
\end{align*}
and define $\projc$ through \eqref{projO}. We
only show that $\projp$ is the orthogonal
projection from $\ltd d$ onto $\hardyp$. The rest of the proof follows
the same steps as the one of Theorem \ref{thm:L2 decomposition in}. 

The operator $\projp$ is self-adjoint, since $\projt$ and $\projtp$
clearly are, and $\left(\restri\tro\troc\right)^{-1}$ is self-adjoint
by Lemma \ref{lem:RBB operator}. Further, 
\begin{equation}
\projp\projp=\projt+\projtp C\projtp C\projtp.\label{eq:P-plus idempotence}
\end{equation}
The idempotence will follow, once we establish that $C\projtp C\projtp=C\projtp$
.

First we show that $\projtp C\projtp=C\projtp$. To see this, observe from Lemma \ref{lem:S1 constant} and the definition of $\troc$ that
$\troc1=\troc\consti$. For $f\in\ltd d$, we then get from the definition of $C$ that
\begin{equation}
\left\langle C\projtp f,\troc1\right\rangle 
=\left\langle C\projtp f,\troc\consti\right\rangle 
=\left\langle \projtp f,\troc\consti\right\rangle
=\left\langle \projtp f,\troc1\right\rangle 
=0.\label{eq:orthogonalty of C}
\end{equation}
Thus, $\projt C\projtp f=0$ and $\projtp C\projtp f=\left(\mathds{1}-\projt\right)C\projtp f=C\projtp f$. 

Next, we prove that $CC\projtp=C\projtp$. For this, observe that
(\ref{eq:orthogonalty of C}) yields $\left\langle \tro C\projtp f,1\right\rangle =0,$
so that $\tro C\projtp f$ has zero mean. Hence, 
\begin{equation}
CC\projtp=\troc\left(\restri\tro\troc\right)^{-1}\left(\restri\tro\troc\right)\left(\restri\tro\troc\right)^{-1}\tro\projtp=C\projtp.
\end{equation}
Consequently, $C\projtp C\projtp=C\projtp$, and (\ref{eq:P-plus idempotence})
says that $\projp\projp=\projp$. 

As to the range of $\projp$, it is easy to see that $\mathcal{N}\left[\projp\right]=\mathcal{N}\left[\tro\right]$, and from Theorem
\ref{thm:null space for PO} we get $\mathcal{N}\left[\tro\right]=\mathcal{N}\left[\poto\right]$.
Thus, by Lemma \ref{lem:adjoint of BO}, we arrive at $\mathcal{R}\left[\projp\right]=\mathcal{N}\left[\poto\right]^{\perp}=\hardyp.$
\end{proof}

\section{Skew-orthogonal decompositions of fields}\label{sec:skew}

In this section we prove two skew-orthogonal decompositions of $\ltd d$.
The first is the Hardy-Hodge decomposition of $\ltd d$. The second
is the splitting of $\ltd d$ into the spaces $\Iod$, $\iOd$, and $\divf$.

We begin with the Hardy-Hodge decomposition, which is a simple consequence
of the above analysis. 
\begin{thm}
\label{thm:hardy-hodge decomposition}The space $\ltd d$ splits into
a not necessarily orthogonal direct sum as 
\begin{align}
\ltd d & =\hardyp+\hardym+\divf.\label{eq:HH}
\end{align}
\end{thm}
\begin{proof}
Note from the definition of $\tric$ and $\troc$ that for
every $\varphi\in\lt$: 
\begin{align}
\left(\tric-\troc\right)\varphi & =\eta\varphi & \frac{1}{2}\left(\tric+\troc\right)\varphi & =\eta\dlay^{\star}\varphi+\nablat\slay\varphi.\label{eq:Hardy-Hodge decomposition}
\end{align}
Now, let $f$ be in $\ltd d$. We use the Hodge decomposition from
Lemma \ref{thm:Hodge decomposition on Lipschitz boundary} and write 
$f=\projDp f+\projD f$, with $\projDp f = \eta f_{\eta} + \projg f_{\tcomp}$. Since $\projg f_{\tcomp}$ is a gradient field, there exists a function $\varphi\in\lt$
such that $\projDp f=\eta f_{\eta}+\nablat\slay\varphi$. Choosing
\begin{align}
h_{+} & =-\troc\left(f_{\eta}-\left(\frac{1}{2}+\dlay^{\star}\right)\varphi\right)\in\hardyp, \label{eq:hp}\\
h_{-} & =\hphantom{-}\tric\left(f_{\eta}+\left(\frac{1}{2}-\dlay^{\star}\right)\varphi\right)\in\hardym, \label{eq:hm}
\end{align}
and using (\ref{eq:Hardy-Hodge decomposition}), we obtain
\begin{multline*}
h_{+}+h_{-}=\left(\tric-\troc\right)\left(f_{\eta}-\dlay^{\star}\varphi\right)+\frac{1}{2}\left(\tric+\troc\right)\varphi=\eta f_{\eta}+\nablat\slay\varphi=\projDp f.
\end{multline*}
The desired decomposition is then given by $f= h_{+}+h_{-}+\projD f$. 

For uniqueness, we first show  that $\hardym\cap\hardyp=\left\{ 0\right\} $.
To see this, assume that $g$ is in $\hardym\cap\hardyp$. Then, there
exist two functions $\phi,\chi\in\lt$ such that,
\begin{equation}
\tric\phi=\troc\chi= g.
\end{equation}
The normal and the tangent components of that equation read
\begin{align}
\left(\frac{1}{2}+K^{\star}\right)\phi & =-\left(\frac{1}{2}-K^{\star}\right)\chi & \text{and }\quad\nablat\slay\phi & =\nablat\slay\chi.
\end{align}
The equation for the tangent part says that $\phi$ differs from $\chi$
by the null space of $\nablat\slay$. By Remark \ref{rem:null-space-of-K-S},
the null spaces of $\left(\frac{1}{2}-K^{\star}\right)$ and $\nablat\slay$
are the same and thus the equation for the normal component reads
\begin{equation}
\left(\frac{1}{2}+\dlay^{\star}\right)\phi=-\left(\frac{1}{2}-K^{\star}\right)\phi,
\end{equation}
leading to $\phi=0$. It follows that $g=\tri\phi=0.$

Now, assume that $\projDp f= g_{-}+g_{+}$, for some $g_{-}\in\hardym$
and $g_{+}\in\hardyp$, then we have that
\begin{align*}
\projDp f = h_{+}+h_{-} & = g_{+}+g_{-}.
\end{align*}
Hence, the two functions $g_{-}-h_{-}$ and $g_{+}-h_{+}$ are in
$\hardym\cap\hardyp$, and thus, $h_{-}=g_{-}$ and $h_{+}=g_{+}$.
\end{proof}
Next, we prove the skew-orthogonal decomposition involving $\Iod$ and $\iOd$.

\begin{lem}\label{lem:exclusively_silent}
It holds that $\mathcal{N}[\poti] \cap \mathcal{N}[\poto]=\divf$.
\end{lem}
\begin{proof}
Let $h$ be in $\mathcal{N}[\poti] \cap \mathcal{N}[\poto]$. Then, by \eqref{eq:null space of inner potential} and \eqref{eq:null space of outer potential} we have that $\tro h=0=\tri h$,
therefore
\begin{equation}
0=\left(\tro+\tri\right)h= h_{\eta}.
\end{equation}
By Corollary \ref{cor:tangent part deteremind by the normal}, the normal component uniquely determines $h$ up to a divergence-free field, and thus, $h\in\divf$.
\end{proof}

\begin{thm}\label{thm:ahh}
The space $\ltd d$ splits into a not-necessarily orthogonal direct sum as 
\begin{align}
\ltd d & =\Iod + \iOd+\divf. \label{eq:IOD}
\end{align}
\end{thm}
\begin{proof}
We will show that for every $f\in\ltd d$ there exists a function $h \in\Iod$
such that $f-h \in\mathcal{N}\left[\poto\right]$, and since $\mathcal{N}\left[\poto\right]=\iOd\oplus\divf$ the decomposition will follow as $f=h+ \projDp(f-h) + \projD(f-h)$.

Recall that a function
$g$ is in $\mathcal{N}\left[\poto\right]$ if and only if $\tro(g)=0$.
By Theorem \ref{thm:null space for PI}, every $h\in\Iod$ satisfies
\begin{equation}
\left(\frac{1}{2}+K\right)h_{\eta}=-\left(\nablat\slay\right)^{\star}h_{\tcomp},
\end{equation}
and thus, we require in view of \eqref{eq:definition of B-1} that 
\begin{equation}
\tro\left(f-h\right)=-\left(\frac{1}{2}-K\right)f_{\eta}+\left(\nablat\slay\right)^{\star}f_{\tcomp}+h_{\eta}=0.
\end{equation}
Choose $h_{\eta}=\left(\frac{1}{2}-K\right)f_{\eta}-\left(\nablat\slay\right)^{\star}f_{\tcomp}$. By the proof of Corollary \ref{cor:tangent part deteremind by the normal}, we can find  $h_{\tcomp}$ such that $h=\eta h_{\eta}+h_{\tcomp}$ is in $\Iod$ and $f-h$ is in $\mathcal{N}\left[\poto\right]$ by construction.

As in the previous theorem,
uniqueness follows since $\Iod\cap\iOd=\left\{ 0\right\}$. 
This concludes the proof.
\end{proof}

\section{Special case---the sphere}\label{sec:sphere}

In this section we show that if $\bo$ is a sphere, then the decompositions in \eqref{eq:HmID}, \eqref{eq:HpOD}, \eqref{eq:HH}, and \eqref{eq:IOD} become equal. To begin, we prove the following theorem. 

\begin{thm}\label{thm:HardyOrto}
The Hardy-spaces are orthogonal if and only if $\bo$ is a sphere.
\end{thm}

It is important to keep in mind that $\o$ is bounded, otherwise the statement is not true and a half-space whose boundary is a hyperplane is a counterexample
(although a half-space is arguably a very large ball). The  theorem is known for Clifford-analytic Hardy spaces \cite[thm. 1.1]{mitrea09}, which
comprise Clifford algebra-valued functions. In contrast to this, the present Hardy-spaces are smaller as they can be seen as vector valued Clifford-analytic functions only. Therefore, the above statement requires proof.

\begin{proof}
It is well known that Hardy-spaces on a sphere are orthogonal (for example, \cite{atfeh10}). Thus, we only have to prove the opposite.

Assume that the Hardy-spaces are orthogonal. Then, by Lemma \ref{lem:adjoint of BI} and Lemma \ref{lem:adjoint of BO} we have for every $\varphi,\chi\in\lt$,
\begin{align}
0=\left\langle\tric\varphi,\troc\chi\right\rangle=\left\langle\tro\tric\varphi,\chi\right\rangle,
\end{align}
which implies the operator identity $\tro \tric = 0$. Then, $\tri \troc=\left(\tro\tric\right)^{\star}=0$ and $\tro \tric - \tri\troc=0$. Using the definition of operators $\tri$, $\tro$, and their adjoints the latter identity yields
\begin{equation}
\tro \tric - \tri\troc = \dlay - \dlay^{\star}=0.
\end{equation}
Consequently, the double layer potential must be self-adjoint. This, however, can only happen when $\bo$ is a sphere by \cite[thm. 4.23]{mitrea09}. This concludes the proof.
\end{proof}

\begin{cor}\label{cor:IODortho}
$\Iod$ and $\iOd$ are orthogonal if and only if $\bo$ is a sphere.
\end{cor}
\begin{proof}
The spaces $\Iod$ and $\hardym$ are orthogonal by Theorem \ref{thm:L2 decomposition in}, and the spaces $\iOd$ and $\hardyp$ are orthogonal by Theorem \ref{thm:L2 decomposition out}. Moreover, from \eqref{eq:HH} and \eqref{eq:IOD}, we also have $\Iod + \iOd=\projDp\ltd d=\hardym + \hardyp$. Thus, $\Iod$ is orthogonal to $ \iOd$ if and only if $\hardym$ is orthogonal to $\hardyp $, and the assertion follows from Theorem \ref{thm:HardyOrto}.
\end{proof}

\begin{cor}\label{cor:ieh}
If $\bo$ is a sphere then $\Iod = \hardyp$ and $\iOd = \hardym$.
\end{cor}
\begin{proof}
On a sphere $\Iod$ and  $\hardyp$ are both orthogonal to $\hardym$. Moreover, $\Iod \oplus \hardym = \hardyp \oplus \hardym$, by \eqref{eq:HmID} and \eqref{eq:HH}. Hence, we get that $\Iod = \hardyp$. An analogous argument holds for $\iOd$ and $\hardym$.
\end{proof}

By the above corollary it is immediate that \eqref{eq:HmID}, \eqref{eq:HpOD}, \eqref{eq:HH}, and \eqref{eq:IOD} are identical when  $\bo$ is a sphere.

\paragraph{Acknowledgments.} AK and CG have been partially funded by BMWi (Bundesministerium f\"ur Wirtschaft und Energie) within the joint project 'SYSEXPL -- Systematische Exploration', grant ref. 03EE4002B.

\printbibliography

\begin{appendix}
\section{Appendix}
\label{appendix} 
The definition of strongly Lipschitz domains is standard and 
can be found in many textbooks, including \cite{adams03,Grisvard11aa}. Nevertheless, the basic differential-geometric notions on Lipschitz manifolds are not easy to ferret out in the literature. For the convenience of the reader we therefore recall the concepts here.

In the following $\mathbb{B}$ will be an open ball in $\mathbb{R}^{d-1}$, and $U$ will denote a doubly truncated cylinder whose cross-section is $\mathbb{B}$.

\paragraph{Strongly Lipschitz domain.}
We call a bounded region $\o$ in $\mathbb{R}^{d}$ a strongly Lipschitz domain if for each $x\in\partial\Omega$
there is a cylinder $U$, a rigid motion $L$, and a Lipschitz function $\psi\colon\mathbb{B}\to\mathbb{R}$ such that, 
\begin{align*}
U\cap \Omega=\{L(y,t) \, :  y\in \mathbb{B},\ 0\leq t<\psi(y)\}
&&\text{and }&&
U\cap\partial\Omega=\{L(y,\psi(y)) \, :  y\in\mathbb{B}\}.
\end{align*} 
Since $\bo$ is compact, we can cover it with finitely many such cylinders $U_j$ associated with $\mathbb{B}_j$, $L_j$, and $\psi_j$,  
 for $j\in\{1,\dots,N\}$. If in addition, we introduce the projection onto the first $(d-1)$ components as $P_{d-1}\colon\mathbb{R}^d\to\mathbb{R}^{d-1}$ then the maps $\phi_{j}:=P_{d-1}\circ L_j^{-1}\colon U_j\cap\bo\to \mathbb{B}_j$ define a system of charts on $\bo$ with Lipschitz inverse $\phi_j^{-1}\colon\mathbb{B}_j\to\bo\subset\mathbb{R}^d$ given by $\phi_j^{-1}(y)=(y,\psi_j(y))$; this provides us with the bi-Lipschitz change of charts and makes $\bo$ a Lipschitz manifold.

\paragraph{Singular and regular points.}
A point $x\in\bo$ is  called singular if there is a
$j\in\{1,\cdots,N\}$ such that $x\in  U_j$ and
$\phi_j^{-1}$ is not differentiable at $\phi_j(x)$. 
A point which 
is not singular is called regular.
We  denote
the set of regular points  by
$\reg$
and put  $\regb=\phi_j(\reg\cap U_j)$. Since $\phi_j$ is bi-Lipschitz, at regular points, its
derivative, $D\phi_j^{-1}(y)$, is injective there.

The set of singular points has $\sigma$-measure zero on $\bo$.
One can see this as follows: by Rademacher's theorem \cite[thm. 3.2]{evans2015measure} the set of singular points $\mathbb{B}_j\setminus \regb$,
has Lebesgue measure zero for each $j$. But the Lebesgue measure on $\mathbb{R}^{d-1}$ coincides with the $\left(d-1\right)$-dimensional Hausdorff measure, $\mathcal{H}^{d-1}$, and thus
the set of singular points has $\mathcal{H}^{d-1}$-measure zero, which is preserved by Lipschitz functions $\phi_j^{-1}$ \cite{federer2014geometric}.

With the above definition, the set of regular points depends on the atlas. Nevertheless, there is also an intrinsic atlas-independent definition (see for example \cite{BPT}), as the set of those points for which the measure theoretic normal to $\o$ exists.

\paragraph{Tangent space.}
At $x\in U_j\cap\reg$,
we define the tangent space 
$\tan{x}\subset\mathbb{R}^d$  to be
$\mathcal{R}[D\phi_j^{-1}(\phi_j(x))]$. 
Then,
each $X\in \tan{x}$ has  a
local representative  %
 in the chart $(U_j,\phi_j)$, 
which is  the unique vector
$v\in\mathbb{R}^{d-1}$ such that $X=D\phi_j^{-1}(\phi_j(x))v$.
At each regular point the tangent space has dimension $d-1$,
and thus,
we can define 
the outer unit normal $\eta(x)$, oriented such that, for small $t>0$ the vector $x+t\eta(x)$ is in $\o^{\exter}$.

\paragraph{Differentiability on $\bo$.}
A map $g\colon\bo\to\mathbb{R}^k$ is said to be 
differentiable at $x\in\reg$ if
$g\circ\phi_j^{-1}$ is differentiable at $\phi_j(x)$; and if  $v$ is the local representative of $X$ then the derivative
$Dg(x)\colon\tan{x}\to\mathbb{R}^k$ is defined by
$Dg(x)(X)=D(g\circ\phi_j^{-1})(\phi_j(x))v$.  By the chain rule, the derivative is chart independent.

When $g:\bo\to\mathbb{R}$ is 
differentiable at $x\in\reg$, the map 
$Dg(x)$ is a linear form on $\tan{x}$, and thus it can be represented 
as $X\mapsto \left\langle X, Y\right\rangle_{\mathbb{R}^{d}}$ for some
unique vector $Y\in\tan{x}$ that we call the tangential gradient of $g$ at
$x$ and denoted by $\nablat g(x)$.

If $f\colon\bo\to\mathbb{R}$ is a Lipschitz function on $\bo$ than
$f\circ\phi_j^{-1}:\mathbb{B}_j\to\mathbb{R}^d$ is also Lipschitz for each $j$. By Rademacher's theorem, 
$f\circ\phi_j^{-1}$  is differentiable a.e. on $\mathbb{B}_j$
and consequently $f$ is differentiable a.e. on $\bo$. Moreover, the derivatives of a Lipschitz function are uniformly bounded by the Lipschitz constant. 

In fact, we can extend $f$ to a function $F$, defined on a small neighborhood around $\bo$, such that $\nablat f$ reads as the Euclidean gradient of $F$. 
To verify this extend $f$ as follows: for each $j$, define $\tilde{f}_j\colon U_j\to\mathbb{R}$ by $\tilde{f}_{j}(L_j(y,t))=f(L_j(y,\psi_j(y))$.
As the differential of $\tilde{f}_j$ is independent of $t$, it is easily seen that
$\tilde{f}_j$ is differentiable a.e. on $\bo\cap U_j$. Using a smooth partition of 
unity relative to the $U_j$'s, $\Omega$, and $\o^{\exter}$, we can glue the $\tilde{f}_j$'s together into a single Lipschitz map $F:\cup_{j}^{N} U_{j}\to \mathbb{R}^d$. This map has the following properties: its restriction to $\bo$ is $f$; it is differentiable a.e. 
on $\bo$; and there, its Euclidean gradient, $\nabla F$, coincides with $\nablat f$.

\paragraph{Sobolev spaces.}
The Sobolev space $W^{1,2}(\bo)$ is the Hilbert space obtained as the completion of Lipschitz functions with respect to the norm, 
\begin{equation}
\label{defSob}
\|\psi\|_{\sob}=\left(\|\psi\|_{\lt}^2+
\|\nablat\psi\|_{\ltd d}^2
\right)^{1/2}.
\end{equation}
Equivalently, the function $\psi$ is in  $\sob$ if and only if,
for each chart $(U_j,\phi_j)$  the function 
$\psi\circ \phi_j^{-1}$ is in the Euclidean
 Sobolev space $W^{1,2}(\mathbb{B}_j)$; thus $\sob$ agrees
with the space $L_1^2(\bo)$ used in \cite[Def. 1.7]{ver84}.
Each $\psi\in\sob$ has a well defined tangential gradient $\nablat\psi\in\tang\subset\ltd d$, see \cite[thm. 3.17]{adams03}. 

\begin{lem}
  \label{convgrad}
  For $f\in L^2(\bo)$, it holds that $\nabla_TSf$ defined in Lemma \ref{gradTP} is the tangential gradient of $Sf\in W^{1,2}(\bo)$.
\end{lem}
\begin{proof}
  Let $F\in L^2(\bo,\mathbb{R}^d)$ be the nontangential limit of $\nabla\SLay f$ on $\bo$ from inside, which is known to exist a.e.  We will show that the
  tangential component of $F$ is equal to $\nabla_TSf$. The limit of
  $\nabla\SLay f$ from outside can be handled similarly, and in view of
the limiting relation in Lemma \ref{gradTP} this will achieve the proof.
The statement being local, it is enough to proceed in a coordinate cylinder $U_j$, and we may assume for simplicity that it has vertical axis so that
  $\phi_j:U_j\cap\bo\to\mathbb{B}_j$ is the projection onto the first $d-1$ components and $\phi_j^{-1}:\mathbb{B}_j\to U_j\cap\bo$ is given by
  $\phi^{-1}_j(y)=(y,\Psi_j(y))$
  for some Lipschitz function $\Psi_j:\mathbb{B}_j\to\mathbb{R}$.
  Set $e_d:=(0,\cdots,0,1)^t$ and let $C_{\theta_1,e_n}(\xi)$ be a natural cone of aperture $2\theta_1$ at
  $\xi\in U_j\cap\bo$ in the chart $(U_j,\phi_j)$, contained in the cone $C_{\theta,Z(\xi)}(\xi)$ from the regular family of cones we have fixed. Recall that $\theta_1$ can be taken independent of $\xi$.
For $\varepsilon>0$ 
small enough that $(y,\Psi_j(y)-\varepsilon)\in\Omega^+\cap U_j$ when
$y\in \mathbb{B}_j$, the smoothness  of $\SLay f$ in $\Omega^+$  implies
that  $h_\varepsilon(y):=\SLay f(y,\Psi_j(y)-\varepsilon)$
is Lipschitz in $\mathbb{B}_j$ with gradient
$(\nabla h_\varepsilon(y))^t= (\nabla\SLay(y,\Psi_j(y)-\varepsilon))^tD\phi_j^{-1}(y)$.
Let $\varepsilon_k\to0$ and observe that 
$(y,\Psi_j(y)-\varepsilon_k)$ converges to $\xi=(y,\Psi_j(y))$ from within 
$C_{\theta,e_d}(\xi)$, hence $\nabla\SLay f(y,\Psi_j(y)-\varepsilon_k)$
converges for $m_{d-1}$-a.e. $y\in \mathbb{B}_j$ to $F(y,\Psi_j(y))$, while being
dominated pointwise in norm by $(\nabla\SLay f)^M(y,\Psi_j(y))$; here, $m_{d-1}$ indicates $d-1$-dimensional Lebesgue measure. As
\[
\int_{\mathbb{B}_j}\left((\nabla\SLay f)^M(y,\Psi_j(y))\right)^2\sqrt{1+|\nabla\Psi_j|^2}dm_{d-1}=\int_{U_j\cap\bo}\left((\nabla\SLay f)^M(\xi)\right)^2d\sigma(\xi)
  \]
  because the image of $d\sigma$ in local coordinates is
  $\sqrt{1+|\nabla\Psi_j|^2}dm_{d-1}$,
  and since $|\nabla\Psi_j|$ is uniformly bounded, we get that
  ${(\nabla\SLay f)^M}$ composed with $\phi_j^{-1}$ lies in $L^2(\mathbb{B}_j)$.
  Therefore, by dominated convergence,
$\nabla h_{\varepsilon_k}$ converges in $L^2(\mathbb{B}_j)$ to
$(D\phi_j^{-1})^t F\circ\phi_j^{-1}$. Likewise,
$h_{\varepsilon_k}$ converges pointwise a.e. to $Sf(y,\Psi_j(y))$ while being dominated pointwise by $(\SLay f)^M(y,\Psi_j(y))$ that lies in $L^2(\mathbb{B}_j)$,
therefore $h_{\varepsilon_k}\to Sf\circ\phi_j^{-1}$ in $L^2(\mathbb{B}_j)$
by dominated convergence. Hence, if we pick a smooth function $\varphi$ with compact support in $\mathbb{B}_j$ and pass to the limit in the relations
$\int_{\mathbb{B}_j}\partial_{y_i}h_{\varepsilon_k}\varphi=-\int_{\mathbb{B}_j}h_{\varepsilon_k}\partial_{y_i}\varphi$,  we find that 
${Sf\circ\phi_j^{-1}}_{|B(y_0,\delta)}$ lies in $W^{1,2}(\mathbb{B}_j)$ 
with
\begin{equation}
\label{expglim}
\nabla ({Sf\circ\phi_j^{-1}})=
(D\phi_j^{-1})^t F\circ\phi_j^{-1}.
\end{equation}
Since the normal $\eta(\xi)$ is orthogonal to the columns of $D\phi_j^{-1}(\phi_j(\xi))$ that span the tangent pace $\tan{\xi}$, a short computation
shows that \eqref{expglim} is equivalent to saying that the tangential component
of the field $F$ is the gradient of $Sf$.
  \end{proof}
\end{appendix}
\end{document}